\documentclass[12pt, makeidx]{amsart}
\usepackage{amsmath,amsthm,amsopn,amssymb}
\usepackage{array}
\usepackage{enumerate,color}

\newtheorem{theorem}{Theorem}[section]
\newtheorem{lemma}[theorem]{Lemma}

\newtheorem{proposition}[theorem]{Proposition}
\newtheorem{corollary}[theorem]{Corollary}
\newtheorem{example}[theorem]{Example}
\newtheorem{remark}[theorem]{Remark}
\newtheorem*{problem*}{Problem}
\numberwithin{equation}{section}

\def\D{\mathbb{D}}
\def\T{\mathbb{T}}

\newcommand{\clq}{\mathcal{Q}}

\newcommand{\cls}{\mathcal{S}}

\newcommand{\clz}{\mathcal{Z}}
\newcommand{\raro}{\rightarrow}
\newcommand{\vp}{\varphi}

\begin{document}

\title[Model spaces invariant under composition operators]{Finite-dimensional model spaces invariant under composition operators}

\author[Muthukumar, Jaydeb, and Batzorig]{P. Muthukumar, Jaydeb Sarkar, and Batzorig Undrakh}
\address{Department of Mathematics and Statistics \\
  Indian Institute of Technology Kanpur\\
  Kanpur,  Uttar Pradesh -208 016\\
  India} \email{pmuthumaths@gmail.com, muthu@iitk.ac.in}

\address{Indian Statistical Institute,
Statistics and Mathematics Unit \\
 8th Mile, Mysore Road\\
Bangalore 560 059\\ India} \email{jaydeb@gmail.com, jay@isibang.ac.in}

\address{Department of Mathematics \\
  National University of Mongolia \\
  Ulaanbaatar \\
  Mongolia} \email{undrakhbatzorig@gmail.com, batzorig\_u@num.edu.mn}




\begin{abstract}
Finite-dimensional model spaces are quotient spaces of the Hardy space on the open unit disc, determined by finite Blaschke products. Composition operators, on the other hand, act by composing Hardy space functions with analytic self-maps of the open unit disc. Both are classical and well-studied objects in the theory of analytic function spaces. In this paper, we present a complete characterization of finite-dimensional model spaces that are invariant under composition operators. Finite cyclic groups and the prime factorizations of natural numbers play a crucial role in understanding the structure of such invariant subspaces and the associated analytic self-maps.
\end{abstract}

\maketitle

\tableofcontents

\section{Introduction}\label{sec:introduction}

The genesis of this paper lies in the interplay between two natural and widely studied analytic objects: composition operators and finite-dimensional spaces associated with Blaschke products. The latter class of spaces is known as model spaces (more specifically, as finite-dimensional model spaces).

Unbeknownst to Mashreghi and Shabankhah \cite{Jav}, the first and second authors in \cite{Mut} studied this problem in broad generality, yielding abstract results and leaving many questions unresolved even at the level of finite-dimensional model spaces. In this paper, we revisit the results of \cite{Jav} and \cite{Mut} and present them in a more unified and broader framework at the level of finite-dimensional model spaces. The results presented in this paper provide a comprehensive treatment of the subject. This paper also identifies and corrects certain errors in \cite{Jav}, sharpens some of its results and proofs, and answers a question posed therein. Nevertheless, some of the groundwork and underlying notions (such as group-theoretic tools) used in this paper originate from \cite{Jav} (as well as from \cite{Mut} to some extent). In our analysis, we combine finite group-theoretic methods with insights drawn from the prime factorizations of natural numbers, specifically those corresponding to the sizes of the zero sets of finite Blaschke products. We now proceed to introduce the key concepts of the paper.

For each $\lambda$ in the open unit disc $\D = \{z \in \mathbb{C}: |z| < 1\}$, we define the corresponding \textit{Blaschke product} $b_\lambda$ as follows:
\[
b_\lambda(z) = \frac{z - \lambda}{1 - \bar{\lambda} z},
\]
for all $z \in \D$. Given $\lambda_1, \ldots, \lambda_n \in \D$, we define the \textit{finite Blaschke product}
\[
\theta = \prod_{j=1}^{n} b_{\lambda_j}.
\]
The \textit{model space} (or, \textit{finite-dimensional model space} to be more specific) associated with $\theta$ is defined as
\[
\clq_{\theta} = H^2 \ominus \theta H^2,
\]
where $H^2$ denotes the Hardy space on $\D$. It is known that
\[
\text{dim} \clq_\theta = n.
\]
The other key object in our study is the composition operator. Given an analytic self-map $\vp: \D \raro \D$ (in short, $\vp \in \cls(\D)$), the \textit{composition operator} $C_\vp$ is defined on $H^2$ by
\[
C_\vp f = f \circ \vp \qquad (f \in H^2).
\]
It is a classical result that $C_\varphi$ is a bounded linear operator on $H^2$ for all $\varphi\in \mathcal{S}(\mathbb{D})$ \cite[Chapter 1]{Shapiro}. The goal of this paper is to determine finite Blaschke products $\theta$ and analytic self-maps $\vp \in \cls(\D)$ such that
\[
C_\vp \clq_\theta \subseteq \clq_\theta,
\]
That is, the finite-dimensional model spaces $\clq_\theta$ that are invariant under the composition operator $C_\vp$. More specifically, our aim is to study the following object:
\[
D(\clq_\theta) = \left\{\varphi \in \cls(\D): C_\varphi \clq_\theta \subseteq \clq_\theta\right\}.
\]
Given that the elements of $\clq_\theta$ are rational functions, one can consider $\clq_\theta$ as a function space defined on the extended complex plane; that is, the Riemann sphere $\mathbb{C}_\infty = \mathbb{C} \cup \{\infty\}$. This perspective suggests another object similar to $D(\clq_\theta)$, defined as follows:
\[
L(\clq_\theta) = \left\{\varphi: \mathbb{C}_\infty \to \mathbb{C}_\infty
\text{ analytic},  \varphi \not \equiv \infty, \text{ and } C_\varphi \clq_\theta \subseteq \clq_\theta \right\}.
\]
Observe that rational functions are precisely the functions holomorphic at $\infty$; hence, every element of $L(\clq_\theta)$ is necessarily a rational function (also see \cite[Section 1]{Jav}). The motivation for these objects comes from \cite{Jav} (and also \cite{Mut}). Three natural questions arise immediately:

\begin{enumerate}
\item What are the elements of the set $D(\clq_\theta)$?
\item What type of structure do these elements form?
\item The same questions as (1) and (2) above, but considered in the setting of $L(\clq_\theta)$.
\end{enumerate}

There are many other questions that arise naturally. For instance, in \cite{Jav}, Mashreghi et al. posed the problem of classifying finite Blaschke products $\theta$ for which
\[
L(\clq_\theta)\neq \{z\},
\]
where $z$ denotes the identity map. Of course, the same question makes sense in the case of $D(\clq_\theta)$: for which Blaschke products $\theta$ does one have $D(\clq_\theta)\neq \{z\}$? We address all of these questions and show that the answers vary from case to case, as is typical in the theory of composition operators. In particular, as we will see, the results naturally divide into two classes: those corresponding to finite Blaschke products that vanish at the origin and those that do not. Following \cite{Jav}, we also discuss the group structure associated with them. As
\[
z \in L(\clq_\theta) \cap D(\clq_\theta),
\]
it follows that $L(\clq_\theta)$ and $D(\clq_\theta)$ are always nonempty. In fact, for any finite Blaschke product $\theta$ other than rotations (see Remark \ref{remark11} for rotations), we have (see Proposition \ref{prop: D sub L})
\[
D(\clq_\theta) \subseteq L(\clq_\theta).
\]
It is easy to see that $L(\clq_\theta)$ and $D(\clq_\theta)$ are semigroups. We now outline some of the main results of this paper: We begin by considering the case where $\theta$ is a finite Blaschke product with a single zero of arbitrary multiplicity:
\[
\theta = b_{\lambda}^n,
\]
for some $\lambda \in \mathbb{D} \setminus \{0\}$ and $n \in \mathbb{N}$. In Proposition \ref{single}, we recollect from \cite[Theorem 2.1]{Jav} that
\[
L(\clq_\theta)= \left\{(1-\bar{\lambda}a)z+a: a \in \mathbb{C}, a\neq \frac{1}{\bar{\lambda}}\right\},
\]
and point out that
\[
D(\clq_\theta)= \left\{z\right\}.
\]
In particular, $D(\clq_\theta)$ is a trivial group, whereas $L(\clq_\theta)$ is a noncyclic infinite group.  At the other extreme, in Theorem \ref{thm: equality}, we prove that for a finite Blaschke product $\theta$ with $\theta(0)\neq 0$ and $D(\clq_\theta)\neq \{z\}$, one has
\[
D(\clq_\theta)=L(\clq_\theta).
\]

For a general finite Blaschke product $\theta$ that is nonvanishing at the origin, the following holds (see Theorem \ref{thm1}): $\varphi \in D(\clq_\theta)$ if and only if there exists a constant $\alpha \in \T$ such that $\varphi(z)=\alpha z$ with
\[
mult_\theta (\lambda) = mult_{\theta}(\bar{\alpha} \lambda),
\]
for all $\lambda \in \clz(\theta)$. For an analytic function $f$ on $X \subseteq \mathbb{C}$, we denote its zero set by
\[
\clz(f)=\left\{\alpha \in X: f(\alpha)=0\right\},
\]
and denote by $mult_f(\alpha)$ the multiplicity of $\alpha$ as a zero of $f$. Note that $mult_f(\alpha)=0$ indicates that $f(\alpha)\neq 0$. The above result also should be attributed to Mashreghi et al. \cite[Corollary 2.4]{Jav}. However, our presentation, proof, and perspective are different. This formulation is also best suited to our framework (see the discussions preceding and following Theorem \ref{thm1}). For instance, from the above, it is now clear that rotations are the appropriate candidates for the set $D(\clq_\theta)$.

A question arises: which subsets of $\T$ give rise to such a set of rotations? Moreover, how can such subsets be related to the finite Blaschke product $\theta$? Within the same setting, we obtain the following answer to this question (see Corollary \ref{coro1}):

\begin{enumerate}
\item $\alpha z \in D(\clq_\theta)$ if and only if $\bar{\alpha} z \in D(\clq_\theta)$.
\item $\alpha z \in D(\clq_\theta)$ if and only if
$mult_\theta (\lambda)= mult_{\theta}(\alpha \lambda)$ for all  $\lambda \in \clz(\theta)$.
\item $\alpha z \not \in D(\clq_\theta)$ if and only if there exists an $\lambda \in \clz(\theta)$ such that $mult_\theta (\lambda)\neq mult_{\theta}(\alpha \lambda)$.
\end{enumerate}

The choice of scalars $\alpha$ in $\T$ above admits a group-theoretic interpretation, which can be stated as follows (see Theorem \ref{thm2}): Assume that
\[
\theta = \prod\limits_{i=1}^n b_{\lambda_i}^{m_i},
\]
for distinct $\{\lambda_1, \ldots, \lambda_n\} \subseteq \mathbb{D} \setminus \{0\}$ and natural numbers $m_i$, $i=1, \ldots, n$. Then $\varphi\in D(\clq_\theta)$ if and only if $\varphi=\omega z$, where
\[
\bar{\omega}=\frac{\lambda_{\sigma(1)}}{\lambda_1}=\frac{\lambda_{\sigma(2)}}{\lambda_2}=\dots=\frac{\lambda_{\sigma(n)}}{\lambda_n},
\]
for some $\sigma \in S_n$ with $m_i=m_{\sigma(i)}$ for all $i \in \{1, \ldots, n\}$. Moreover, in this case, we have
\[
D(\clq_\theta)=\langle e^{\frac{2\pi i}{d}}z \rangle,
\]
for some divisor $d$ of $n =\# \clz(\theta)$ (see Corollary \ref{coro3}). Here we follow the standard notation: We use $\langle e^{\frac{2\pi i}{m}} z \rangle$, $m \in \mathbb{N}$, to denote the finite cyclic group
\[
\left\{e^{\frac{2\pi i t}{m}} z: t = 0, 1, \ldots, m-1\right\},
\]
under composition generated by $e^{\frac{2\pi i}{m}} z$ (note that $ e^{\frac{2\pi i}{m}}$ is the primitive $m$-th root of unity). We also denote by $S_n$ the symmetric group on $n$ letters.

This raises a number of natural questions, many of which have been both posed and addressed in this paper. One problem we highlight here is the following: given $\theta$ as above, under what conditions do we have
\[
D(\clq_\theta) = \langle e^{\frac{2\pi i}{n}}z\rangle?
\]
In Theorem \ref{thm4}, we prove that $D(\clq_\theta) = \langle e^{\frac{2\pi i}{n}}z\rangle$ if and only if
\[
\left\{\dfrac{\lambda_1}{\lambda_1}, \dfrac{\lambda_2}{\lambda_1}, \ldots, \dfrac{\lambda_n}{\lambda_1}\right\},
\]
is a multiplicative group, and
\[
m_1=m_2=\dots=m_n.
\]

We now turn to the question of determining when $L(\clq_\theta)=\{z\}$ as well as $D(\clq_\theta)=\{z\}$. We provide the following answer (see Theorem \ref{thm7}): Let $\theta$ be a finite Blaschke product. Assume that $\theta(0)\neq 0$. Consider the prime factorization of $n := \# \clz(\theta)$ as
\[
n=p_1^{k_1} \dots p_m^{k_m}.
\]
Then
\[
D(\clq_\theta)=\{z\},
\]
if and only if for each $j \in \{1, \ldots, m\}$, there exists $\lambda_j \in \clz(\theta)$ such that
\[
mult_\theta (\lambda_j)\neq mult_{\theta}\left(e^{\frac{2\pi i}{p_j}} \lambda_j\right).
\]
Moreover, in Theorem \ref{thm - triv L}, we prove the following: Given a finite Blaschke product $\theta$, we have $L(\clq_\theta)=\{z\}$ if and only if the following conditions hold:
\begin{enumerate}
\item $\theta(0)\neq 0$.
\item $\#\clz(\theta) \geq 2$.
\item For every pair of constants $(a,b)$ other than $(1,0)$, there exists $\lambda \in \clz(\theta)$ such that
\[
mult_\theta (\lambda)\neq mult_{\theta}\left(\frac{\bar{a} \lambda}{1-\bar{b} \lambda}\right).
\]
\end{enumerate}

The situation changes when we shift our focus to finite Blaschke products that vanish at the origin. To illustrate this, we outline the following results, as observed in Corollary \ref{cor: z b lambda}: For $\lambda \in \mathbb{D} \setminus \{0\}$ and $n\geq 1$, define
\[
\theta = z b_{\lambda}^n.
\]
Then
\[
L(\clq_\theta)= \left\{\vp \in  \text{Mob}(\mathbb{C}_\infty): \vp\left(\frac{1}{\bar{\lambda}}\right) = \frac{1}{\bar{\lambda}} \right\} \cup
\left(\mathbb{C} \setminus \left\{\frac{1}{\bar{\lambda}}\right\}\right),
\]
and
\[
D(\clq_\theta)= \left\{\vp \in \cls(\D) \cap  \text{Mob}(\mathbb{C}_\infty): \vp\left(\frac{1}{\bar{\lambda}}\right) = \frac{1}{\bar{\lambda}} \right\} \cup \mathbb{D}.
\]
These results are comparable to Theorem 2.5 and Corollary 2.6 of \cite{Jav}. Moreover, in the above setting, we have new information:
\[
D(\clq_\theta)  \cap  \text{Aut}(\mathbb{D}),
\]
is uncountable, and $L(\clq_\theta)\setminus \mathbb{C}$ is a non-cyclic infinite group. Here, $ \text{Mob}(\mathbb{C}_\infty)$ denotes the group of all M\"{o}bius transformations of the extended complex plane $\mathbb{C}_\infty = \mathbb{C} \cup \{\infty\}$.

However, for a finite Blaschke product with a zero at the origin of higher multiplicity, we encounter a markedly different scenario once again (see Theorems \ref{thm m geq 2 lambda} and \ref{thm m geq 2 lambda-2}): Let $m\geq 2$, $n\geq 1$ be natural numbers and let $\lambda \in \D \setminus \{0\}$. Define
\[
\theta = z^m b_\lambda^n.
\]
If $m-1\neq n$, then
\[
L(\clq_\theta)= \left\{(1-\bar{\lambda}a)z+a: a\neq \frac{1}{\bar\lambda} \right\} \cup
\left(\mathbb{C} \setminus \left\{\frac{1}{\bar{\lambda}}\right\}\right),
\]
and
\[
D(\clq_\theta) =  \left\{z\right\} \cup \D.
\]
If $m-1= n$, then
\[
L(\clq_\theta)= \left\{(1-\bar{\lambda}a)z+a:a\neq \frac{1}{\bar{\lambda}} \right\} \cup
\left\{\frac{c-z}{1-\bar{\lambda}z}:c\neq \frac{1}{\bar{\lambda}} \right\} \cup
 \left(\mathbb{C} \setminus \left\{\frac{1}{\bar{\lambda}}\right\}\right),
\]
and
\[
D(\clq_\theta) =L(\clq_\theta)\cap \cls(\D) =  \left\{z\right\} \cup \left\{\frac{\lambda-z}{1-\bar{\lambda}z} \right\} \cup \D.
\]
In particular, for $\theta = z^m b_\lambda^n$, $m\geq 2$, $n\geq 1$, we have the following (see Corollary \ref{cor: m 2 n 1 lambda}):
\begin{enumerate}
\item $L(\clq_\theta)^*$ is an uncountable group.
\item If $m-1\neq n$, then $D(\clq_\theta)^*$ is a trivial group.
\item If $m-1=n$, then $D(\clq_\theta)^*$ is a cyclic group of order $2$.
\end{enumerate}
In the above and what follows, $L(\clq_\theta)^*$ and $D(\clq_\theta)^*$ refer to the sets of all nonconstant functions in $L(\clq_\theta)$ and $D(\clq_\theta)$, respectively.

Given a M\"{o}bius transformation $\varphi(z)=\dfrac{az+b}{cz+d}$, we construct another M\"{o}bius transformation $\tilde{\vp}$ as (see Section \ref{sec: mob})
\[
\tilde{\varphi}(z)=\dfrac{\bar{a}z-\overline{c}}{-\bar{b}z+\bar{d}}.
\]
Now we consider a finite Blaschke product $\theta$ satisfying
\[
\# (\clz(\theta)\setminus \{0\}) \geq 2,
\]
and record the following three results:

\begin{enumerate}
\item[(i)] Suppose $\theta(0)=0$ and $\theta'(0)\neq 0$. Then Theorem \ref{lem2.7} implies that $\varphi \in L(\clq_\theta)^*$  if and only if $\varphi$ is a M\"obius transformation and\[
mult_\theta(\lambda)=mult_\theta(\tilde{\varphi}(\lambda)),
\]
for all $\lambda \in \clz(\theta)\setminus \{0\}$. Moreover, we have
\[
D(\clq_\theta)^*=L(\clq_\theta)^* \cap \mathcal{S}(\mathbb{D}).
\]
\item[(ii)] Suppose $mult_\theta(0) \geq 2$. Then Theorem \ref{bm} implies that $\varphi \in L(\clq_\theta)^*$ if and only if $\varphi$ is an affine transformation\begin{equation*}
\varphi(z)=\alpha z+\beta,
\end{equation*}
for some scalars $\alpha(\neq 0)$ and $\beta$ such that
\[
mult_\theta(\lambda)=mult_\theta(\tilde{\varphi}(\lambda)) \mbox{~ for all~ } \lambda \in \clz(\theta)\setminus \{0\}.
\]
\item[(iii)] Suppose $\theta(0) = 0$, and let $\varphi\in D(\clq_\theta)^*$. Then Corollary \ref{lary1} implies the following:
\begin{enumerate}
\item $\varphi \in  \text{Aut}(\mathbb{D})$.
\item There exists $n \in \mathbb{N}$ such that  $\underbrace{\varphi \circ \dots \circ \varphi}_{n \text{ times }} = z$.
\end{enumerate}
\end{enumerate}

The most general result concerning finite Blaschke products $\theta$ that vanish at the origin is Corollary \ref{cor: general group}, which states that if
\[
\#(\clz(\theta)\setminus \{0\}) \geq 2,
\]
then $L(\clq_\theta) \setminus \mathbb{C}$ and $D(\clq_\theta) \setminus \mathbb{D}$ form a group under composition. In the more specific case, when
\[
\# (\clz(\theta)\setminus \{0\}) = 2,
\]
we have the following observation (see Theorem \ref{thm: 2 zero L uncount}): Let $\alpha$ and $\beta$ be two distinct nonzero elements of $\mathbb{D}$ and $m,n\in \mathbb{N}$. Let $b_{\alpha,\beta}$ denote the unique disc automorphism that interchanges $\alpha$
and $\beta$. Define
\[
\theta=zb_\alpha^m b_\beta^n.
\]
Then the following two properties hold:
\begin{enumerate}
\item $L(\clq_\theta)^*$ is an uncountable group.
\item  $D(\clq_\theta)^* =
\begin{cases}
\{z\} & \text{ if } m\neq n \\
\{z, b_{\alpha,\beta}\} & \text{otherwise.}
\end{cases}$
\end{enumerate}

This observation identifies an incorrect claim in \cite[Theorem 2.8]{Jav}, namely, that the set $L(\clq_\theta)$ is a finite cyclic group.

Needless to say, the above collection of results varies significantly from case to case, highlighting the rich and intricate structure of composition operators, even within the framework of finite-dimensional model spaces. Many additional results in this paper as well as in the broader literature further explore these themes along similar lines.

From this perspective, we remind the reader that the invariant subspace problem for composition operators is a classical and challenging problem. In fact, the invariant subspace problem for operators on Hilbert spaces is equivalent to the one-dimensional minimal invariant subspace problem for composition operators with hyperbolic symbols \cite{NRW}. This scenario also serves as an additional motivation besides the study of lattice structures of composition operators for the theory developed in this paper. We refer the reader to \cite{ISCO,Inv2014,Inv2005,Inv2015,Inv2010,Mut} and references therein for more results in this direction. For a list of exotic properties and their connections to diverse aspects of composition operators, we refer the reader to \cite{Burdon,  Fricain, Eva}.

The remaining part of the paper is structured as follows. Section \ref{sec: prep} presents some general observations that are used throughout the paper. Section \ref{sec: nonzero} provides a precise description of model spaces corresponding to finite Blaschke products that do not vanish at the origin. Section \ref{sec: cyclic} explores the natural emergence of finite cyclic groups, particularly in the context of finite subsets of the unit circle $\T$. In Section \ref{sec: nontrivial}, we present a complete solution to the question of the nontriviality of the groups that arise naturally in the study of self-analytic functions on $\D$ within the framework of model spaces. Sections \ref{sec: not vanish} and \ref{sec: mob} deal with the analysis of model spaces associated with finite Blaschke products that vanish at the origin. This setting brings M\"{o}bius transformations into consideration. Section \ref{sec: example} is devoted entirely to a detailed example, motivated by earlier work in \cite{Jav}. We revisit this example to correct certain inaccuracies and to present its full significance within a more general framework. The final section, Section \ref{sec: concluding}, presents general observations, outlines potential directions for future research, and includes a summary table highlighting some of the key results obtained in this paper.

\section{Basic observations}\label{sec: prep}

We treat this section as a warm-up for the results presented in the forthcoming sections. We also derive some elementary observations. Recall that
\[
\dim \mathcal{Q}_\theta = \deg \theta,
\]
for each finite Blaschke product $\theta$. We first consider the case of one-dimensional model spaces. These spaces correspond to $\theta = b_\alpha$, $\alpha \in \D$. However, here we focus only on the case $\alpha = 0$ (see Proposition \ref{single} for the case $\alpha \neq 0$):

\begin{remark}\label{remark11}
Let $\theta = \alpha z$ for some $\alpha \in \T$. Then
\[
\clq_\theta = H^2 \ominus \alpha z H^2 = \mathbb{C},
\]
where $\mathbb{C}$ represents the space of all constant functions. In this case, it is trivial to note that
\[
L(\clq_\theta)= \{\text{all the rational functions}\},
\]
and
\[
D(\clq_\theta)= \mathcal{S}(\mathbb{D}),
\]
and consequently, $L(\clq_\theta)$ and $D(\clq_\theta)$ are not comparable semigroups.
\end{remark}

We recall a general fact about the basis vectors of finite-dimensional model spaces, which will be used frequently in what follows: If $\theta$ is a finite Blaschke product with distinct zeros $\lambda_1, \ldots, \lambda_k$ of multiplicities $n_1,\ldots, n_k$, respectively, then the set
\begin{equation}\label{eqn: basis all}
\Big\{c_{\lambda_i}^{(\ell_i)}: 0\leq \ell_i\leq n_i-1, 1\leq i \leq k \Big\},
\end{equation}
form a basis for $\clq_\theta$. Here, the function $c_\lambda^{(s)}$ is defined by
\[
c_\lambda^{(s)}(z)=\frac{z^s}{(1-\bar{\lambda}z)^{s+1}}.
\]
If $\lambda\neq 0$, then $c_\lambda^{(s)}$ may also be taken as
\begin{equation}\label{eqn: basis special}
c_\lambda^{(s)}(z)=\frac{1}{(1-\bar{\lambda}z)^{s+1}}.
\end{equation}
In particular, we have
\[
\dim \clq_\theta = n_1 + \cdots + n_k.
\]

The following result is key, relying on the fact that $\clq_\theta$ consists of rational functions. While the containment is elementary, it will be useful in what follows.

\begin{proposition}\label{prop: D sub L}
Let $\theta$ be a finite Blaschke product that is not a rotation. Then
\[
D(\clq_\theta) \subseteq L(\clq_\theta).
\]
\end{proposition}
\begin{proof}
Suppose $\varphi \in D(\clq_\theta)$. To show that $\vp \in L(\clq_\theta)$, it suffices to show that $\vp$ is a rational map. First, assume that $z \in \clq_\theta$. Then
\[
z \circ \varphi= \varphi \in \clq_\theta,
\]
which implies that $\varphi$ is a rational function. Next, assume that $z\notin \clq_\theta$. Since $\theta$ is not a rotation, it follows that
\[
\clz(\theta) \setminus \{0\} \neq \emptyset.
\]
Pick $\lambda\in \clz(\theta)\setminus\{0\}$. This implies $\theta = b_\lambda \tilde{\theta}$ for some finite Blaschke product $\tilde{\theta}$. Then $\dfrac{1}{1-\bar{\lambda} z} \in \clq_\theta$, and hence
\[
\dfrac{1}{1-\bar{\lambda} z} \circ \varphi=  \dfrac{1}{1-\bar{\lambda} \varphi} \in \clq_\theta,
\]
is a rational function. It follows, in this case as well, that $\varphi$ is a rational map. Therefore, in both cases we have $\varphi \in L(\clq_\theta)$. This completes the proof.
\end{proof}

The following simple lemma will be used throughout this paper repeatedly.

\begin{lemma}\label{lem1}
Let $\varphi : \D \raro \mathbb{C}$ be an analytic function. Suppose
\begin{equation*}
\frac{1}{1-p\varphi(z)}=\frac{a}{1-qz} \qquad (z \in \D),
\end{equation*}
for some  complex numbers $0<|p|\leq |q|<1$ and $a\neq 0$. Then the following conditions are equivalent:
\begin{enumerate}
\item $\varphi \in \cls(\D)$.
\item $a=1$ and $|p|=|q|$.
\item $\varphi$ is a rotation.

\end{enumerate}
In either of these cases, we have $\varphi(z)=\dfrac{q}{p} z$.
\end{lemma}
\begin{proof}
By cross-multiplying the equation given in the hypothesis, we immediately obtain
\[
\varphi(z)=\frac{q}{ap}z+\frac{a-1}{ap}.
\]
Suppose (1) holds, that is, $\varphi \in \cls(\D)$. Assume, if possible, that $a\not =1$. Since $0<|p|<1 $, it follows that
\begin{equation*}
\left|\frac{a-1}{ap}\right|> \left|\frac{a-1}{a}\right|.
\end{equation*}
By the triangle inequality, we have
\begin{equation*}
\left|\frac{a-1}{ap}\right|+ \left|\frac{1}{a}\frac{q}{p}\right|> \left|\frac{a-1}{a}\right|+
\left|\frac{1}{a}\right| \geq 1.
\end{equation*}
Therefore $\varphi$ is not a self-map of $\mathbb{D}$ (recall that $\alpha z + \beta$, $\alpha\neq 0$, is a self-map of $\D$ if and only if $|\alpha|+|\beta| \leq 1$). It yields that $a=1$, and thus
\[
\varphi(z)=\frac{q}{p} z.
\]
Also note that $\varphi$ cannot be a self-map of $\mathbb{D}$ if $|p|<|q|$. Hence
$\varphi(z)=\frac{q}{p} z$ with $|p|=|q|$. Thus, we have shown that (1) implies (2). The implications that (2) implies (3) and (3) implies (1) are immediate.
\end{proof}

\section{Blaschke products nonvanishing at $0$}\label{sec: nonzero}

This section discusses the structures of $D(\clq_\theta)$ and $L(\clq_\theta)$ under the condition that $\theta$ is a finite Blaschke product and that $\theta(0)$ is nonzero. Part of this section also recalls a collection of results, primarily from \cite{Jav} (and also from \cite{Mut}), to provide a complete structures of $D(\clq_\theta)$ and $L(\clq_\theta)$ associated with $\theta$. However, we note that, along the way, we will also refine some of those results to provide a clearer understanding of these sets.

We first consider the case where $\theta$ is a finite Blaschke product with a singleton zero set. A major part of the following result was established in \cite{Jav} and \cite{Mut}. The second part follows easily from Lemma \ref{lem1}; see also the concluding paragraph of Subsection 2.1 in \cite{Jav}. We will use this result in the later sections of the paper.

\begin{proposition}\label{single}
Let $\theta = b_{\lambda}^n$ for some $\lambda \in \mathbb{D} \setminus \{0\}$ and $n \in \mathbb{N}$. Then
\begin{equation*}
L(\clq_\theta)= \left\{(1-\bar{\lambda}a)z+a: a \in \mathbb{C}, a\neq \frac{1}{\bar{\lambda}}\right\},
\end{equation*}
and
\begin{equation*}
D(\clq_\theta)= \left\{z\right\}.
\end{equation*}
\end{proposition}
\begin{proof}
For the first part, see \cite[Theorem 2.1]{Jav} (or see \cite[Theorem 3.1]{Mut}). For the representation of $D(\clq_\theta)$, we observe that there exists a constant $c$ (see the final part of the proof of \cite[Theorem 3.1]{Mut}) such that
\[
1-\bar{\lambda}\varphi=c(1-\bar{\lambda}z).
\]
By Lemma \ref{lem1}, we get $\varphi(z)=z$, which proves that $D(\clq_\theta)= \left\{z\right\}$.
\end{proof}

In particular, we obtain the following contrasting result:

\begin{corollary}
Let $\theta = b_{\lambda}^n$ for some $\lambda \in \mathbb{D} \setminus \{0\}$ and $n \in \mathbb{N}$. Then $D(\clq_\theta)$ is a trivial group, whereas
$L(\clq_\theta)$ is a noncyclic infinite group.
\end{corollary}

This completes the discussion of finite Blaschke products with a singleton zero set. We now turn to the case of finite Blaschke products $\theta$ where $\theta(0)\neq 0$ and
\[
\#\clz(\theta) \geq 2.
\]
The following lemma is an improvement of \cite[Lemma 2.2]{Jav}. We do, however, utilize that lemma to produce the subsequent sharper version, wherein all relevant conditions are essentially consolidated into a single, unified condition.

\begin{lemma}\label{lem: Javad}
Let $\theta$ be a finite Blaschke product. Assume that $\theta(0) \neq 0$ and $\#\clz(\theta) \geq 2$. Then $\varphi \in\ L(\clq_\theta)$ if and only if
there exist constants $a (\neq 0)$ and $b$ such that
\[
\varphi(z)=az+b,
\] with
\[
\text{mult}_\theta (\lambda) = \text{mult}_{\theta}\left(\frac{\overline{a} \lambda}{1-
\overline{b} \lambda}\right),
\]
for all $\lambda \in \clz(\theta)$.
\end{lemma}
\begin{proof}
Suppose $\varphi \in\ L(\clq_\theta)$. By the definition of $L(\mathcal{Q}_\theta)$, the function $\varphi$ is rational. Now if $\varphi=z$, then the desired result follows. On the other hand, if $\varphi\neq z$, then \cite[Lemma 2.2]{Jav} implies the result. For the converse direction, suppose there exist constants $a (\neq 0)$ and $b$ such that
\[
\varphi(z)=az+b,
\] with
\[
\text{mult}_\theta (\lambda) = \text{mult}_{\theta}\left(\frac{\overline{a} \lambda}{1-
\overline{b} \lambda}\right),
\]
for all $\lambda \in \clz(\theta)$. To show $\varphi \in\ L(\clq_\theta)$ (that is, $ C_\varphi \clq_\theta \subseteq \clq_\theta$), we start with an arbitrary basis element $\dfrac{1}{(1-\bar{\lambda}z)^{s}}$ of $\clq_\theta$, where $\lambda\in \clz(\theta)$ and $s\leq {mult}_\theta (\lambda)$. Then
\[
\frac{1}{(1-\bar{\lambda}z)^{s}}\circ \varphi = \frac{1}{(1-\bar{\lambda}(az+b))^{s}} = \frac{1}{(1-
b \bar{ \lambda})^s}\frac{1}{\left(1-\overline{\left(\frac{\overline{a} \lambda}{1-
\overline{b}\lambda}\right)}\,z\right)^s}.
\]
Combined with the fact that
\[
s\leq \text{mult}_\theta (\lambda) = \text{mult}_{\theta}\left(\frac{\overline{a} \lambda}{1-
\overline{b} \lambda}\right),
\]
we obtain
\[
\frac{1}{(1-\bar{\lambda}z)^{s}}\circ \varphi \in \clq_\theta,
\]
which completes the proof.
\end{proof}

The above multiplicity condition also ensures that $\dfrac{\overline{a} \lambda}{1- \overline{b} \lambda} \in \clz(\theta)$. Next, we recall from \cite[Theorem 2.3]{Jav} a general result concerning model spaces associated with Blaschke products having two distinct zeros. Given a self-map $f$ and a natural number $n$, we write
\[
f^{[n]}= \underbrace{f \circ \dots \circ f}_{n \text{ times }}.
\]

\begin{proposition}[]\label{affine}
Let $\theta$ be a finite Blaschke product. Assume that $\theta(0)\neq 0$ and $\# \clz(\theta) \geq 2$. Then there exist scalars $a \neq 0$ and $b$ and natural number $n \in \mathbb{N}$ such that
\[
L(\clq_\theta)= \left\{ z, \varphi, \varphi^{[2]}, \dots, \varphi^{[n-1]} \right\}
\]
where
\[
\varphi(z)=az+b,
\]
and
\[
\varphi^{[n]}(z) =z.
\]
In particular, $L(\clq_\theta)$ is a finite cyclic subgroup of $\mathrm{Aut}(\mathbb{C})$.
\end{proposition}

In these circumstances, we also have a description of $D(\clq_\theta)$ from \cite[Corollary 2.4 ]{Jav}: Let $\theta$ be a finite Blaschke product. Suppose $\theta(0)\neq 0$ and $\# \clz(\theta) \geq 2$. Then $\varphi \in D(\clq_\theta)$ if and only if following conditions hold:
\begin{itemize}
\item[(i)] $\varphi(z)= \alpha z$ for some $\alpha \in \T$.
\item[(ii)] $\alpha^n=1$ for some $n\in \mathbb{N}$.
\item[(iii)] $\alpha\lambda\in \clz(\theta)$ for all $\lambda \in \clz(\theta)$.
\item[(iv)] The zeros $\{\lambda, \alpha \lambda,\ldots, \alpha^{n-1}\lambda\}$ of $\theta$ have same multiplicity.
\end{itemize}

The proof of this result is involved, as it relies on non-trivial results concerning the iterative behavior of loxodromic and parabolic M\"{o}bius transformations. In the following, we present a slightly modified version of \cite[Corollary 2.4]{Jav} analogous to a conjugation of the scalar part in rotation maps along with a different proof. This version leads to several useful consequences. Moreover, it will both imply and be implied by \cite[Corollary 2.4]{Jav}, as we will point out after the proof.

\begin{theorem}\label{thm1}
Let $\theta$ be a finite Blaschke product. Assume that $\theta(0)\neq 0$. Then $\varphi \in D(\clq_\theta)$ if and only if $\varphi(z)=\alpha z$ for some $\alpha \in \T$ with
\[
mult_\theta (\lambda) = mult_{\theta}(\bar{\alpha} \lambda),
\]
for all $\lambda \in \clz(\theta)$.
\end{theorem}
\begin{proof}
Suppose $\theta = \prod\limits_{i=1}^n b_{\lambda_i}^{m_i}$, where $\lambda_i$ are nonzero distinct elements in $\mathbb{D}$ and $m_i$ are natural numbers. Suppose that
$\varphi \in D(\clq_\theta)$. If $n=1$, then the result simply follows from Proposition \ref{single} (with the choice of $\alpha = 1$). Assume that $n\geq 2$. By Propositions \ref{prop: D sub L} and \ref{affine}, there exist constants $\alpha (\neq 0)$ and $\beta$ such that
\[
\varphi(z)=\alpha z+\beta.
\]
Choose $\lambda \in \clz(\theta)$ such that $|\lambda|=\min\left\{|\lambda_1|, \ldots, |\lambda_n|\right\}$. As $\dfrac{1}{1-\bar{\lambda}z} \in \clq_\theta$, we have
\[
\dfrac{1}{1-\bar{\lambda}z}\circ \varphi=\dfrac{1}{1-\bar{\lambda} \varphi}=\sum\limits_{i=1}^{n}\sum\limits_{j=1}^{m_i}
\dfrac{c_{ij}}{(1-\bar{\lambda}_i z)^j},
\]
for some constants $c_{ij}$ with at least one of them is nonzero. By the identity theorem, the equality holds on the entire complex plane except possibly at finitely many poles. Since the function on the left-hand side has exactly one pole of order $1$, the above equation implies that
$$
\dfrac{1}{1-\bar{\lambda}\varphi}=\dfrac{c}{1-\bar{\lambda}_k z},
$$
for some $k \in \{1, \ldots, n\}$ and $c\neq 0$. At this point, Lemma \ref{lem1} applies and yields
\[
\varphi(z)=\alpha z,
\]
for some constant $\alpha \in \T$. Fix $i \in \{1, \ldots, n\}$. As $\dfrac{1}{(1-\bar{\lambda}_i z)^{m_i}}\in \clq_\theta$, we have
\[
\frac{1}{(1-\bar{\lambda}_i \alpha z)^{m_i}} = \frac{1}{(1-\bar{\lambda}_i z)^{m_i}} \circ \vp \in \clq_\theta,
\]
and hence $\bar{\alpha}\lambda_i\in \clz(\theta)$. Indeed, if $\bar{\alpha}\lambda_i \notin \clz(\theta)$, then, by considering a basis for $\clq_\theta$ consisting of functions of the form given in \eqref{eqn: basis special}, we obtain
\[
\frac{1}{(1-\bar{\lambda}_i \alpha z)^{m_i}} \notin \clq_\theta,
\]
leading to a contradiction. Thus, we have that $\lambda \in \clz(\theta)$ implies $\bar{\alpha}\lambda \in \clz(\theta)$ with
\[
m_i = mult_{\theta}( \lambda_i) \leq mult_\theta (\bar{\alpha}\lambda_i).
\]
By applying this argument iteratively, we conclude, for all $\lambda \in \clz(\theta)$ and $k\in \mathbb{N}$, that
\[
\bar{\alpha}^k\lambda \in \clz(\theta),
\]
and
\[
mult_{\theta}( \lambda)\leq mult_\theta (\bar{\alpha}\lambda)\leq \cdots \leq mult_\theta (\bar{\alpha}^k\lambda).
\]
Moreover, as $\theta$ has only finitely many zeros, it follows that
\begin{equation}\label{eqn: alpha m = 1}
\bar{\alpha}^m=1,
\end{equation}
for some $m \in \mathbb{N}$, and then, for any $\lambda \in \clz(\theta)$, we have
\[
mult_{\theta}( \lambda)\leq mult_\theta (\bar{\alpha}\lambda)\leq \cdots \leq mult_\theta (\bar{\alpha}^m\lambda) = mult_\theta (\lambda),
\]
and hence, $mult_\theta (\lambda)=mult_{\theta}(\bar{\alpha} \lambda)$.

\noindent The converse part is easy: Suppose $\alpha$  is a constant such that $\varphi(z)=\alpha z$ and $mult_\theta (\lambda)= mult_{\theta}(\bar{\alpha} \lambda)$ for all $\lambda \in \clz(\theta)$. Then
for any arbitrary basis element $\dfrac{1}{(1-\bar{\lambda}z)^{l}}$ of $\clq_\theta$, for some $l \in \{1, \ldots, mult_{\theta}(\lambda)\}$, we have
\[
\frac{1}{(1-\bar{\lambda}z)^{l}}\circ \varphi = \frac{1}{(1-\bar{\lambda} \alpha z)^{l}} \in \clq_\theta,
\]
and hence $\varphi \in D(\clq_\theta)$. This completes the proof of the theorem.
\end{proof}

It is important to note that the identity $mult_\theta (\lambda)=mult_{\theta}(\bar{\alpha} \lambda)$ for all $\lambda \in \clz(\theta)$ in the above result can be rewritten as
\[
mult_\theta (\lambda)=mult_{\theta}(\lambda/\overline{\alpha}),
\]
for all $\lambda \in \clz(\theta)$. Since $|\alpha|=1$ (because $\bar{\alpha}^m=1$ for some $m\in \mathbb{N}$), the above condition is equivalent to
\[
mult_\theta (\lambda)=mult_{\theta}(\alpha \lambda),
\]
for all $\lambda \in \clz(\theta)$. This recovers the exact statement of \cite[Corollary 2.4 ]{Jav}. Similarly, it also recovers the version we proved above. As an immediate consequence of these results, we have the following:

\begin{corollary}\label{coro1}
Let $\theta$ be a finite Blaschke product and let $\alpha \in \T$. Assume that $\theta(0)\neq 0$. Then we have the following:
\begin{enumerate}
\item $\alpha z \in D(\clq_\theta)$  if and only if $\bar{\alpha} z \in D(\clq_\theta)$.
\item $\alpha z \in D(\clq_\theta)$  if and only if
$mult_\theta (\lambda)= mult_{\theta}(\alpha \lambda)$ for all  $\lambda \in \clz(\theta)$.
\item $\alpha z \not \in D(\clq_\theta)$ if and only if there exists an $\lambda \in \clz(\theta)$ such that $mult_\theta (\lambda)\neq mult_{\theta}(\alpha \lambda)$.
\end{enumerate}
\end{corollary}

Since we have been working with two sets, $L(\clq_\theta)$ and $D(\clq_\theta)$, from the beginning, one may naturally ask whether these sets are equal. While there is evidence that they are often not closely related, aside from the fact that $D(\clq_\theta)$ is contained in $L(\clq_\theta)$, up to a rotation of $\theta$ (see Proposition \ref{prop: D sub L}), we show in the following that, when $\theta$ does not vanish at the origin and $D(\clq_\theta)$ is nontrivial, the two sets are indeed equal.

\begin{theorem}\label{thm: equality}
Let $\theta$ be a finite Blaschke product. Suppose $\theta(0)\neq 0$ and $D(\clq_\theta)\neq \{z\}$. Then
\[
D(\clq_\theta)=L(\clq_\theta).
\]
\end{theorem}
\begin{proof}
As pointed out above, by Proportion \ref{prop: D sub L}, we already know that $D(\clq_\theta) \subseteq L(\clq_\theta)$. In view of Proposition \ref{single}, we know that
\[
m:=\#\clz(\theta) \geq 2.
\]
Indeed, if $m=1$, then by Proposition \ref{single}, $D(\clq_\theta)$ is trivial, which contradicts the hypothesis. Hence, $m \geq 2$, which establishes the claim. By Proposition \ref{affine}, there exist $a, b \in \mathbb{C}$ and $n \in \mathbb{N}$ such that $\varphi(z)=az+b$ and
\[
L(\clq_\theta)= \left\{ z, \varphi, \varphi^{[2]}, \dots, \varphi^{[n-1]} \right\},
\]
and $\varphi^{[n]}=z$. Since $D(\clq_\theta)\neq \{z\}$, by Theorem \ref{thm1}, there exists $\alpha\neq 1$ such that $\alpha z\in D(\clq_\theta)$. As a consequence of Proposition \ref{prop: D sub L}, $\varphi^{[k]}=\alpha z$ for some $k\leq n-1$. At this point, based on the coefficient $b$, we consider two cases:

\noindent \textit{Case I:} Suppose $b\neq 0$. If $a=1$, then $L(\clq_\theta)$ contains the infinite set
\[
\{\varphi=z+b, \varphi^{[2]}=z+2b, \varphi^{[3]}=z+3b, \ldots\},
\]
which is not possible, as $L(\clq_\theta)$ is a finite set. Thus, $a\neq 1$, and therefore
\[
\varphi^{[k]}=a^kz+(a^{k-1}+a^{k-2}+\cdots+1)b=a^kz+\left(\frac{a^k-1}{a-1}\right)b=\alpha z.
\]
Comparing coefficients, we obtain $a^k-1=0$, and consequently $\alpha=1$, which leads to a contradiction to the fact that $\alpha\neq 1$.

\noindent \textit{Case II:} Suppose $b = 0$. In this case, $\varphi(z)=az$. Lemma \ref{lem: Javad} implies
\[
mult_\theta (\lambda) = mult_{\theta}(\bar{a} \lambda),
\]
for all $\lambda \in \clz(\theta)$. Fix $\lambda \in \clz(\theta)$. Then, $\bar{a} \lambda, \bar{a}^2 \lambda,\ldots$ are zeros of the finite Blaschke product  $\theta$, which forces $a^m=1$ for some $m\in \mathbb{N}$. Using Theorem \ref{thm1}, we conclude that $\varphi(z)=az\in D(\clq_\theta)$, and consequently, we obtain
\[
L(\clq_\theta)= \left\{ z, \varphi, \varphi^{[2]}, \dots, \varphi^{[n-1]} \right\}\subseteq D(\clq_\theta).
\]
This completes the proof.
\end{proof}

Results concerning representations of functions in $D(\clq_\theta)$ and $L(\clq_\theta)$, corresponding to finite Blaschke products $\theta$ that vanish at the origin, will be considered in Section \ref{sec: not vanish}.

\section{Cyclic groups}\label{sec: cyclic}

Here, we continue with the setting of the previous section and apply elementary group-theoretic tools to study the sets $D(\clq_\theta)$ and $L(\clq_\theta)$.

In the setting described in Theorem \ref{thm1}, we know that $\alpha$ is of finite order-specifically,
\[
\alpha^m = 1,
\]
as observed in equation \eqref{eqn: alpha m = 1} (and also in \cite[Corollary 2.4]{Jav}, as previously noted). We now turn to providing a precise interpretation of the index $m$  and its role in the structure of $D(\clq_\theta)$ and $L(\clq_\theta)$. This is done through the lens of finite cyclic groups. To that end, we first present another characterization of functions in the class $D(\clq_\theta)$. For each $n \in \mathbb{N}$, in what follows, we write
\[
J_n = \{1, \ldots, n\},
\]
and denote by $S_n$ the symmetric group of degree $n$.

\begin{theorem}\label{thm2}
Let $\{\lambda_1, \ldots, \lambda_n\} \subseteq \mathbb{D} \setminus \{0\}$ be a set of distinct elements and let $m_i \in \mathbb{N}$, $i=1, \ldots, n$. Let $\theta = \prod\limits_{i=1}^n b_{\lambda_i}^{m_i}$. Then $\varphi\in D(\clq_\theta)$ if and only if $\varphi=\omega z$, where
\[
\bar{\omega}=\frac{\lambda_{\sigma(1)}}{\lambda_1}=\frac{\lambda_{\sigma(2)}}{\lambda_2}=\dots=\frac{\lambda_{\sigma(n)}}{\lambda_n},
\]
for some $\sigma \in S_n$ with $m_i=m_{\sigma(i)}$ for all $i \in J_n$.
\end{theorem}
\begin{proof}
Let $\varphi \in D(\clq_\theta)$. By Theorem \ref{thm1}, there exists $\omega \in \T$ such that $\varphi(z) = \omega z$ and $\bar{\omega}\lambda\in \clz(\theta)$ for all $\lambda \in \clz(\theta)$. Thus for each $i \in J_n$, there is an $j\in J_n$ such that $\bar{\omega}\lambda_i= \lambda_j$. Define $\sigma : J \to J$ by
\[
\overline{\omega}\lambda_i= \lambda_{\sigma(i)},
\]
for all $i\in J_n$. It is easy to see that $\sigma$ is injective. Since $J_n$ is finite set, $\sigma$ is onto. Thus $\sigma \in S_n$ and
\begin{equation*}
\overline{\omega}=\frac{\lambda_{\sigma(1)}}{\lambda_1}=\frac{\lambda_{\sigma(2)}}{\lambda_2}=\dots=\frac{\lambda_{\sigma(n)}}{\lambda_n}.
\end{equation*}
Since $\omega z\in D(\clq_\theta)$, again by Theorem \ref{thm1}, for each $i\in J_n$, we have
\[
m_i=mult_\theta (\lambda_i)=mult_{\theta}(\overline{\omega} \lambda_i)=mult_\theta (\lambda_{\sigma(i)})=m_{\sigma(i)}.
\]
For the converse direction, assume $\vp = \omega z$, where $\omega$ satisfies the identity given in the statement. Then
\begin{equation*}
C_\vp\left( \frac{1}{(1-\bar{\lambda}_i z)^k}\right) = \frac{1}{(1-\bar{\lambda}_i z)^k} \circ \varphi= \frac{1}{(1-\bar{\lambda}_i \omega z)^k}=\frac{1}{(1-\overline{\lambda_{\sigma(i)}} z)^k} \in \clq_\theta,
\end{equation*}
for all $k \in \{1, \ldots, m_i\}$ and $i \in J_n$. Hence $\varphi \in D(\clq_\theta)$, completing the proof of the theorem.
\end{proof}

\begin{remark}\label{rem1}
In the setting of Theorem \ref{thm2}, pick $\omega z\in D(\clq_\theta)$ and $\sigma \in S_n$. Clearly
\begin{equation*}
\omega^n=\frac{\overline{\lambda_{\sigma(1)}}}{\overline{\lambda_1}}\cdot \frac{\overline{\lambda_{\sigma(2)}}}{\overline{\lambda_2}}\cdots \frac{\overline{\lambda_{\sigma(n)}}}{\overline{\lambda_n}}=1.
\end{equation*}
Moreover, if $\sigma$ contains a cycle $(i_1 i_2 \dots i_k)$, that is, if $\sigma(i_1)=i_2, \sigma(i_2)=i_3, \ldots, \sigma(i_{k-1})=i_k$, and $\sigma(i_k)=i_1$, then
\begin{equation*}
\omega^k=\frac{\overline{\lambda_{\sigma(i_1)}}}{\overline{\lambda_{i_1}}}\cdot \frac{\overline{\lambda_{\sigma(i_2)}}}{\overline{\lambda_{i_2}}}\cdots \frac{\overline{\lambda_{\sigma(i_k)}}}{\overline{\lambda_{i_k}}}=1.
\end{equation*}
\end{remark}

Recall that for each $n\in \mathbb{N}$, we use $\langle e^{\frac{2\pi i}{n}} z \rangle$ to denote the finite cyclic group
\[
\left\{e^{\frac{2\pi i t}{n}} z: t = 0, 1, \ldots, n-1\right\},
\]
under composition generated by $e^{\frac{2\pi i}{n}} z$, where $ e^{\frac{2\pi i}{n}}$ is the primitive $n$-th root of unity. The above remark yields the following corollary:

\begin{corollary}\label{coro2}
Let $\theta$ be a finite Blaschke product with $\theta(0)\neq 0$. If $n := \# \clz(\theta)$, then
\begin{equation*}
D(\clq_\theta)\subseteq \langle e^{\frac{2\pi i}{n}} z \rangle.
\end{equation*}
\end{corollary}

We can say a little more about $D(\clq_\theta)$:

\begin{theorem}\label{thm3}
Let $\theta$ be a finite Blaschke product with $\theta(0)\neq 0$. If $n := \# \clz(\theta)$, then $D(\clq_\theta)$ is a cyclic subgroup of $\langle e^{\frac{2\pi i}{n}}z \rangle$.
\end{theorem}
\begin{proof}
Let $\varphi \in D(\clq_\theta)$. By Theorem \ref{thm1}, there exists $\omega \in \T$ such that $\varphi = \omega z$. By part (1) of Corollary \ref{coro1}, we have
\[
\frac{1}{\omega}z=\overline{\omega} z \in D(\clq_\theta).
\]
Note that $\overline{\omega} z$ is the inverse of $\omega z$ under composition. Hence $D(\clq_\theta)$ is a subgroup of the cyclic group $\langle e^{\frac{2\pi i}{n}}z \rangle$.
\end{proof}

We now recall the general fact that the order of any subgroup of a finite cyclic group divides the order of the group. Therefore, we obtain the following corollary, which is comparable to the analogue of $L(\clq_\theta)$ in \cite[Theorem 2.3]{Jav}:

\begin{corollary}\label{coro3}
Let $\theta$ be a finite Blaschke product. Assume that $\theta(0)\neq 0$. Then
\begin{equation*}
D(\clq_\theta)=\langle e^{\frac{2\pi i}{d}}z \rangle,
\end{equation*}
for some divisor $d$ of $\# \clz(\theta)$.
\end{corollary}

As an example, we consider the following:

\begin{example}
Suppose $\theta$ is a finite Blaschke product with four distinct zeros and $\theta(0)\neq 0$. Then there are exactly three possible cyclic groups for $D(\clq_\theta)$:
\[
\{z\}, \langle -z\rangle =\{z, -z\}, \text{ and } \langle iz\rangle= \{z, -z, iz, -iz\}.
\]
\end{example}

Given an integer $n$ and a divisor $d$ of $n$, one expects that there exists a finite Blaschke product $\theta$ such that $\# \clz(\theta) = n$ and $D(\clq_\theta)=\langle e^{\frac{2\pi i}{d}}z\rangle$. This is indeed the case:

\begin{theorem}
Let $n \in \mathbb{N}$, and let $d$ be a divisor of $n$. Then there exists a Blaschke product $\theta$ such that $n := \# \clz(\theta)$ and
\begin{equation*}
D(\clq_\theta)=\langle e^{\frac{2\pi i}{d}}z\rangle.
\end{equation*}
\end{theorem}
\begin{proof}
Let $md=n$. Fix scalars $0<r_1< \cdots< r_m<1$, and consider the finite Blaschke product
\[
\theta = \prod_{t=1}^m \prod_{s=1}^d b_{\alpha^s r_t}^{t},
\]
where $\alpha= e^{\frac{2\pi i}{d}}$. Clearly, $\# \clz(\theta) = n$, and
\begin{equation*}
\clz(\theta) = \left\{\alpha^sr_t: 1 \leq s \leq d, 1 \leq t \leq m\right\},
\end{equation*}
and
\[
mult_\theta(\alpha^sr_t) = t,
\]
for all $1\leq s \leq d$. In particular, $\theta$ has exactly $n (=d m)$ zeros, all located on the union of the $m$ circles $|z|=r_j$, $j=1, \ldots, m$. Label all zeros of $\theta$ as $\lambda_1, \ldots, \lambda_n$ so that
\[
\lambda_1=r_1\alpha, \ldots, \lambda_d=r_1\alpha^d,
\]
and then
\[
\lambda_{d+1} = r_2\alpha, \ldots, \lambda_{2d}=r_2\alpha^d,
\]
and so on. Accordingly, the first $d$ zeros have multiplicity $1$, the next $d$ zeros have multiplicity $2$, and so forth. Suppose $\varphi \in D(\clq_\theta)$, that is, $C_\varphi(\clq_\theta) \subseteq \clq_\theta$. By Theorem \ref{thm2}, there exist $\beta \in \T$ and $\sigma \in S_n$ such that $\varphi(z)=\beta z$ and
\[
\bar{\beta} = \dfrac{\lambda_{\sigma(i)}}{\lambda_i},
\]
and $mult_\theta(\lambda_i) = mult_\theta(\lambda_{\sigma(i)})$ for all $i=1, \ldots, n$. In particular, this implies that the restriction of $\sigma$ acts as a permutation on each of the sets $\{1, \ldots, d\}$, $\{d+1, \dots, 2d\}$, and so on. In particular,
\begin{equation*}
\bar{\beta}^d=\frac{\lambda_{\sigma(1)}}{\lambda_1}\cdot \frac{\lambda_{\sigma(2)}}{\lambda_2} \cdots \frac{\lambda_{\sigma(d)}}{\lambda_d}=1,
\end{equation*}
that is, $\beta^d=1$. As $\beta \in \left\{1, \alpha, \ldots, \alpha^{d-1}\right\}$, it follows that $D(\clq_\theta)\subseteq \langle \alpha z \rangle$. For the reverse inclusion, note that from our construction it is clear that
\[
mult_\theta (\lambda)= mult_{\theta}(\alpha \lambda),
\]
for all  $\lambda \in \clz(\theta)$. Then, by part (2) of Corollary \ref{coro1}, it follows that $\alpha z \in D(\clq_\theta)$. Since $D(\clq_\theta)$ is a semi-group, we conclude that
\[
\{\alpha z, \alpha^2 z, \dots, \alpha^{d-1}z, \alpha^dz=z\} \subseteq D(\clq_\theta),
\]
and hence $\langle \alpha z \rangle \subseteq D(\clq_\theta)$. This completes the proof.
\end{proof}

Given a finite Blaschke product $\theta$ with $\theta(0)\neq 0$, by Corollary \ref{coro2}, we know that $D(\clq_\theta) \subseteq \langle e^{\frac{2\pi i}{n}}z\rangle$, where $\# \clz(\theta) = n$. It is natural to ask, when do we have
\[
D(\clq_\theta) = \langle e^{\frac{2\pi i}{n}}z\rangle?
\]
In the following, we answer to this question. Before that, we set up a notation: Given a set of $n$ distinct points $\{\lambda_1, \ldots, \lambda_n\} \subseteq \mathbb{D} \setminus \{0\}$, we define
\[
\Lambda_n = \left\{\dfrac{\lambda_1}{\lambda_1}, \dfrac{\lambda_2}{\lambda_1}, \dots, \dfrac{\lambda_n}{\lambda_1}\right\}.
\]
Note that a rearrangement of $\{\lambda_1, \ldots, \lambda_n\}$ may yield a different set $\Lambda_n$. However, one may fix the labeling to redefine $\Lambda_n$, and such relabeling will not affect the following result.

\begin{theorem}\label{thm4}
Let $\{\lambda_1,\ldots, \lambda_n\}$ be a set of $n$ distinct points in $\mathbb{D} \setminus \{0\}$ and let $m_i \in \mathbb{N}$, $i=1, \ldots, n$. If
\[
\theta = \prod_{i=1}^{n} b_{\lambda_i}^{m_i},
\]
then
\[
D(\clq_\theta) = \langle e^{\frac{2\pi i}{n}}z\rangle,
\]
if and only if $\Lambda_n$ is a multiplicative group, and
\[
m_1=m_2=\dots=m_n.
\]
\end{theorem}
\begin{proof}
Let $\alpha= e^{\frac{2\pi i}{n}}$. First, we prove the sufficiency part. Since all the $\lambda_i$'s are distinct, the multiplicative group $\Lambda_n$ has order $n$. Consequently, every element of $\Lambda_n$ must be an $n$-th root of unity. This implies that
\[
\Lambda_n = \left\{1, \alpha, \ldots, \alpha^{n-1}\right\},
\]
and hence
\[
\{\lambda_1,\lambda_2,\ldots,\lambda_n\} = \{\lambda_1,\alpha\lambda_1, \ldots,\alpha^{n-1}\lambda_1\}.
\]
By part (2) of Corollary \ref{coro1}, we know that $e^{\frac{2\pi i}{n}}z \in D(\clq_\theta)$, which implies $\langle e^{\frac{2\pi i}{n}}z \rangle\subseteq  D(\clq_\theta)$ (recall that $D(\clq_\theta)$ is semigroup). The reverse inclusion, $D(\clq_\theta)\subseteq \langle e^{\frac{2\pi i}{n}}z \rangle$, is due to Corollary \ref{coro2}. This proves that $\langle e^{\frac{2\pi i}{n}}z \rangle = D(\clq_\theta)$. For the necessary part, assume that $D(\clq_\theta)=\langle \alpha z \rangle$. As $\alpha z\in D(\clq_\theta)$, by part (2) of Corollary \ref{coro1}, we have
\[
\{\alpha\lambda_1, \ldots, \alpha^{n-1}\lambda_1\} \subseteq \clz(\theta),
\]
and the multiplicities of all these zeros are the same as that of $\lambda_1$. Consequently, $\Lambda_n = \left\{1, \alpha, \ldots, \alpha^{n-1}\right\}$, which completes the proof of the theorem.
\end{proof}

We know from Corollary \ref{coro3} that $D(\clq_\theta)=\langle e^{\frac{2\pi i}{d}}z \rangle$ for some divisor $d$ of $n$. This leads us to the following question: Given a finite Blaschke product $\theta$ not vanishing at the origin and a divisor $d$ of $n:= \# \clz(\theta)$, when does
\[
D(\clq_\theta)=\langle e^{\frac{2\pi i}{d}}z\rangle?
\]
The following result provides an answer. Here, we make use of the prime factorization of $n = \# \clz(\theta)$.

\begin{theorem}
Let $\theta$ be a finite Blaschke product and let $d$ be a divisor of $n:= \# \clz(\theta)$. Suppose $\theta(0)\neq 0$ and let
\[
n = p_1^{k_1} p_2^{k_2} \cdots p_m^{k_m},
\]
be the prime factorization of $n$. Then
\[
D(\clq_\theta)=\langle e^{\frac{2\pi i}{d}}z\rangle,
\]
if and only if
\[
mult_\theta (\lambda)= mult_{\theta}\left(e^{\frac{2\pi i}{d}} \lambda\right),
\]
for all $\lambda \in \clz(\theta)$, and for each $j \in \{1, \ldots, m\}$, there exist $\lambda_j \in \clz(\theta)$ such that
\[
mult_\theta (\lambda_j)\neq mult_{\theta}\left(e^{\frac{2\pi i}{dp_j}} \lambda_j\right).
\]
\end{theorem}
\begin{proof}
Suppose $D(\clq_\theta)=\langle e^{\frac{2\pi i}{d}}z\rangle$. Then $e^{\frac{2\pi i}{d}}z \in D(\clq_\theta)$ and $e^{\frac{2\pi i}{d p_j}}z\not \in D(\clq_\theta)$ for all $j=1,\ldots, m$. The necessary part then follows from Corollary \ref{coro1}. For the reverse direction, assume that $mult_\theta (\lambda)= mult_{\theta}(e^{\frac{2\pi i}{d}} \lambda)$ for all $\lambda \in \clz(\theta)$. Part (2) of Corollary \ref{coro1} implies that
\[
e^{\frac{2\pi i}{d}}z \in D(\clq_\theta),
\]
and hence $\langle e^{\frac{2\pi i}{d}}z\rangle \subseteq D(\clq_\theta)$. On the other hand, by Theorem \ref{coro3}, there exists a divisor $\ell$ of $n$ such that
\[
D(\clq_\theta)=\langle e^{\frac{2\pi i}{\ell}}z\rangle.
\]
Suppose $\ell \neq d$. Since $ e^{\frac{2\pi i}{d}}z \in \langle e^{\frac{2\pi i}{\ell}}z\rangle$, there exist $k \in \{1, \ldots, \ell\}$ such that
\[
e^{\frac{2\pi i}{d}}= e^{\frac{2\pi i k}{\ell}},
\]
which implies
\[
\frac{1}{d}-\frac{k}{\ell} \in \mathbb{Z}.
\]
As $d\geq 1$ and $\frac{k}{\ell} \in (0,1]$, it follows that
\[
\frac{1}{d}-\frac{k}{\ell} \in (-1, 1),
\]
which yields
\[
\frac{1}{d}-\frac{k}{\ell}=0,
\]
equivalently, $\ell= kd$. In particular, $d$ is a divisor of $\ell$. Since $\ell \neq d$, it follows that
\[
dp_j | \ell,
\]
for some $j=1,\ldots,m$. Then
\[
e^{\frac{2\pi i}{d p_j}}z \in \langle e^{\frac{2\pi i}{\ell}}z\rangle=D(\clq_\theta),
\]
and hence, by part (2) of Corollary \ref{coro1}, we have $mult_\theta (\lambda)= mult_{\theta}\left(e^{\frac{2\pi i}{dp_j}} \lambda\right)$ for all $\lambda \in \clz(\theta)$. This contradicts to the other condition that for each $j \in \{1, \ldots, m\}$, there exist $\lambda_j \in \clz(\theta)$ such that
\[
mult_\theta (\lambda_j)\neq mult_{\theta}\left(e^{\frac{2\pi i}{dp_j}} \lambda_j\right).
\]
Therefore, we conclude that $\ell=d$, that is, $D(\clq_\theta)=\langle e^{\frac{2\pi i}{d}}z\rangle$. This completes the proof of the theorem.
\end{proof}

A natural question that arises is, what would be the correct analogue of the statement of the above theorem for general Blaschke products (and then for inner functions)? Where this analogy is unclear, one may consider an infinite cyclic group and ask whether there exists a meaningful connection between such a group and infinite Blaschke products (or inner functions). Some more concrete questions along these lines will be outlined at the end of the paper.

\section{On nontrivial groups}\label{sec: nontrivial}

In this section, we provide a complete solution to the question of the nontriviality of both $L(\clq_\theta)$ and $D(\clq_\theta)$. In a sense, the nontriviality of these spaces is closely related, as already observed in Proposition \ref{prop: D sub L}, where we have
\[
D(\clq_\theta)\subseteq L(\clq_\theta).
\]
We begin by addressing an equivalent formulation of the problem namely, determining when $D(\clq_\theta)$ is trivial. We proceed from simple cases to the most general ones. As we will see, the results for all cases, from particular to general, differ. Recall that given a set $\{\lambda_1, \ldots, \lambda_n\}$ of $n$ distinct points in $\mathbb{D} \setminus \{0\}$, we define
\[
\Lambda_n = \left\{\dfrac{\lambda_1}{\lambda_1}, \dfrac{\lambda_2}{\lambda_1}, \dots, \dfrac{\lambda_n}{\lambda_1}\right\},
\]
and regard it as a multiplicative group whenever we wish to view it as a group. The following result comes directly from existing results:

\begin{proposition}\label{thm5}
Let $n$ be a prime number, $\{\lambda_1, \ldots, \lambda_n\} \subseteq \mathbb{D} \setminus \{0\}$ be a set of $n$ distinct scalars, and let $\{m_i\}_{i=1}^n \subseteq \mathbb{N}$. If
\[
\theta = \prod_{i=1}^{n} b_{\lambda_i}^{m_i},
\]
then $D(\clq_\theta) = \{z\}$ if and only if either not all the $m_i$'s are equal or $\Lambda_n$ is not a group.
\end{proposition}
\begin{proof}
Since $n$ is prime, Corollary \ref{coro3} implies that $D(\clq_\theta)$ is either $\{z\}$ or $\langle e^{\frac{2\pi i}{n}}z\rangle$. The equivalent formulation of the proposition now follows directly from Theorem \ref{thm4}.
\end{proof}

Note that $\Lambda_2 = \left\{\lambda_1/\lambda_1, \lambda_2/\lambda_1 \right\}$ is a multiplicative group
if and only if
\[
\lambda_1+\lambda_2=0.
\]
Therefore, we have the following easy consequence: Let $\{\lambda_1,\lambda_2\}$ be a pair of distinct points in $\mathbb{D} \setminus \{0\}$, and let $m_1, m_2 \in \mathbb{N}$. Set
\[
\theta = b_{\lambda_1}^{m_1} b_{\lambda_2}^{m_2}.
\]
Then we have the following:
\begin{enumerate}
\item $D(\clq_\theta)=\{z\}$ if and only if $m_1\neq m_2$ or $\lambda_1+\lambda_2\neq 0$.
\item $D(\clq_\theta)=\{z, -z\}$ if and only if $m_1= m_2$ and $\lambda_1+\lambda_2=0$.
\end{enumerate}

Proposition \ref{thm5} addresses the triviality of $D(\clq_\theta)$ in the case where $n = \# \clz(\theta)$ is a prime number. We now turn to the case of general natural numbers. Our eventual goal is to apply the prime factorization of $n$, and as a first step, we consider the case where $n$ is a prime power:

\begin{proposition}\label{thm6}
Let $\theta$ be a finite Blaschke product. Assume that $\theta(0)\neq 0$ and $\# \clz(\theta) = p^k$ for some prime $p$ and $k \in \mathbb{N}$. Then $D(\clq_\theta)=\{z\}$ if and only if there exists $\lambda \in \clz(\theta)$ such that
\[
mult_\theta (\lambda)\neq mult_{\theta}\left(e^{\frac{2\pi i}{p}} \lambda\right).
\]
\end{proposition}
\begin{proof}
By Theorem \ref{thm3}, we know that $D(\clq_\theta)$ is a subgroup of $\langle e^{\frac{2\pi i}{p^k}} z\rangle$. Since every nontrivial subgroup of $\langle e^{\frac{2\pi i}{p^k}} z\rangle$ contains $e^{\frac{2\pi i}{p}}z$, it follows that $D(\clq_\theta)=\{z\}$ if and only if $e^{\frac{2\pi i}{p}}z \not \in D(\clq_\theta)$. However, by part (3) of Corollary \ref{coro1}, the latter occurs if and only if there exists $\lambda \in \clz(\theta)$ such that $mult_\theta (\lambda)\neq mult_{\theta}(e^{\frac{2\pi i}{p}} \lambda)$.
\end{proof}

We now consider the general case of $n = \# \clz(\theta)$:

\begin{theorem}\label{thm7}
Let $\theta$ be a finite Blaschke product. Assume that $\theta(0)\neq 0$. Consider the prime factorization of $n := \# \clz(\theta)$ as
\[
n=p_1^{k_1} \cdots p_m^{k_m}.
\]
Then $D(\clq_\theta)=\{z\}$ if and only if for each $j \in \{1, \ldots, m\}$, there exists $\lambda_j \in \clz(\theta)$ such that
\[
mult_\theta (\lambda_j)\neq mult_{\theta}\left(e^{\frac{2\pi i}{p_j}} \lambda_j\right).
\]
\end{theorem}
\begin{proof}
Suppose $D(\clq_\theta)=\{z\}$ and fix $j \in \{1, \ldots, m\}$. Since $e^{\frac{2\pi i}{p_j}}z \not \in D(\clq_\theta)$, part (3) of Corollary \ref{coro1} implies
\[
mult_\theta (\lambda_j)\neq mult_{\theta}\left(e^{\frac{2\pi i}{p_j}} \lambda_j\right),
\]
for some $\lambda_j \in \clz(\theta)$. For the converse, we proceed by contradiction. Suppose, for the sake of argument, that $D(\clq_\theta) \neq \{z\}$. By Theorem \ref{coro3}, there exists a divisor $d ( > 1)$ of $n$ such that
\[
D(\clq_\theta)=\langle e^{\frac{2\pi i}{d}}z \rangle.
\]
Thus, there exists a prime factor $p_j$ of $n$ such that $p_j |d$; that is, $d=p_j\ell$ for some $\ell \in \mathbb{N}$. Hence
\[
e^{\frac{2\pi i}{p_j}}z=\left( e^{\frac{2\pi i}{d}}\right)^\ell z \in \langle e^{\frac{2\pi i}{d}}z\rangle= D(\clq_\theta),
\]
that is, $e^{\frac{2\pi i}{p_j}}z \in  D(\clq_\theta)$ for some $j \in\{1, \ldots, m\}$. By part (2) of Corollary \ref{coro1}, it then follows that $mult_\theta (\lambda)=mult_{\theta}(e^{\frac{2\pi i}{p_j}} \lambda)$ for all $\lambda \in \clz(\theta)$, which leads to a contradiction.
\end{proof}

In the setting of the above result, we obtain the following statement by contrapositive, which serves as an answer to the nontriviality question for $D(\clq_\theta)$:
\[
D(\clq_\theta)\neq\{z\},
\]
if and only if there exist a prime factor $p_j$ of $n$ such that
\[
mult_\theta (\lambda)= mult_{\theta}\left(e^{\frac{2\pi i}{p_j}} \lambda\right),
\]
for all $\lambda \in \clz(\theta)$. Moreover, in this case, we have
\[
\langle e^{\frac{2\pi i}{p_j}}z \rangle \subseteq D(\clq_\theta).
\]

\begin{theorem}\label{thm - triv L}
Let $\theta$ be a finite Blaschke product. Then $L(\clq_\theta)=\{z\}$ if and only if the following conditions hold:
\begin{enumerate}
\item $\theta(0)\neq 0$.
\item $\#\clz(\theta) \geq 2$.
\item For every pair of constants $(a,b)$ other than $(1,0)$, there exists $\lambda \in \clz(\theta)$ such that
\[
mult_\theta (\lambda)\neq mult_{\theta}\left(\frac{\bar{a} \lambda}{1-\bar{b} \lambda}\right).
\]
\end{enumerate}
\end{theorem}
\begin{proof}
Suppose $L(\clq_\theta)=\{z\}$. If $\theta(0)=0$, then $\clq_\theta$ contains all constant functions and consequently, $L(\clq_\theta) \supseteq \mathbb{C}$ (recall that here $\mathbb{C}$ refers to the set of all constant functions), which is uncountable and hence not possible. Therefore, $\theta(0)\neq 0$. If $\theta$ has only one zero, then by Proposition \ref{single}, $L(\clq_\theta)$ is again uncountable. Hence, $\#\clz(\theta) \geq 2$. Part (3) follows from Lemma \ref{lem: Javad}. The converse also follows easily from Lemma \ref{lem: Javad}.
\end{proof}

In particular, we have: Let $\theta$ be a finite Blaschke product. Then $L(\clq_\theta)$ is nontrivial if and only if any one of the following holds:
\begin{enumerate}
\item $\theta(0)=0$.
\item There exist $n \in \mathbb{N}$ and a nonzero $\lambda \in \D$ such that $\theta = b_\lambda^n$.
\item $\theta(0)\neq 0$, $\#\clz(\theta) \geq 2$, and there exists a pair of constants $(a,b)$ other than $(1,0)$ such that for all $\lambda \in \clz(\theta)$, we have
\[
mult_\theta (\lambda)= mult_{\theta}\left(\frac{\bar{a} \lambda}{1-\bar{b} \lambda}\right).
\]
\end{enumerate}

We conclude this section with an illustration of Theorem \ref{thm7}.

\begin{example}
Let $n \geq 2$ be a natural number. Fix a finite Blaschke product $\theta$ such that
\[
\# \clz(\theta)=n,
\]
such that
\[
\theta\Big(\frac{3}{4}\Big)=0,
\]
and that its remaining $n-1$ zeros lie on the circle $|z|= \frac{1}{2}$, with arbitrary multiplicities. For any prime divisor $p$ of $n$, by the definition of $\theta$, we have
\[
\frac{3}{4}e^{\frac{2\pi i}{p}}\notin \clz(\theta),
\]
which implies that $mult_{\theta}\left(\frac{3}{4} e^{\frac{2\pi i}{p}}\right)=0$. Hence,
\[
0=mult_{\theta}\left(\frac{3}{4} e^{\frac{2\pi i}{p}}\right)\neq mult_\theta \left(\frac{3}{4}\right) \geq 1,
\]
for all prime divisors $p$ of $n$. Therefore, by Theorem \ref{thm7}, we conclude that $D(\clq_\theta)=\{z\}$.
\end{example}

\section{Blaschke products vanishing at $0$ and $\# \clz(\theta) =2$}\label{sec: not vanish}

Here, we consider model spaces corresponding to finite Blaschke products that vanish at the origin. In this setting, the results often differ from those in the case where the Blaschke product does not vanish at the origin. In this section, we will specifically focus on Blaschke products $\theta$ whose zero set contains exactly two elements. More specifically, we consider the case
\[
\theta = z^m b_{\lambda}^n,
\]
where $\lambda$ is a nonzero element of $\mathbb{D}$ and $m, n \in \mathbb{N}$. Thus, in this case, $\theta$ such that $\theta(0) = 0$ and $\# (\clz(\theta)\setminus \{0\}) =1$. The remaining case, that is,  $\theta(0) = 0$ and $\# (\clz(\theta)\setminus \{0\}) \geq 2$, will be treated in the following section. In what follows, we will observe that the sets $D(\clq_\theta)$ and $L(\clq_\theta)$ may contain constant functions. Keeping this in mind, throughout the paper, for any subset $X \subseteq \mathbb{C}$, when we write
\[
X \subseteq L(\clq_\theta),
\]
we mean that $L(\clq_\theta)$ contains constant functions that assume values in $X$. The same convention will be used for $D(\clq_\theta)$. For simplicity of notation, we write
\[
L(\clq_\theta)^* = L(\clq_\theta) \setminus \mathbb{C} \text{ and } D(\clq_\theta)^* = D(\clq_\theta) \setminus \mathbb{D}.
\]
The following is the $\alpha = 0$ case of \cite[Theorem 3.1]{Mut}, along with its direct consequence:

\begin{proposition}\label{prop: z*m}
Let  $\theta(z) =z^m$, $m\geq 2$. Then
\begin{enumerate}
\item $L(\clq_\theta)= \{az+b: a, b \in \mathbb{C}, a\neq 0\}\cup \mathbb{C}$.
\item $D(\clq_\theta)=\{az+b: a, b \in \mathbb{C}, a\neq 0, |a|+|b|\leq 1 \}\cup \D$.
\item $L(\clq_\theta)^*= \{az+b: a, b \in \mathbb{C}, a\neq 0\}$ is a non-cyclic group.
\item $D(\clq_\theta)^*$ is not a group.
\end{enumerate}
\end{proposition}

As for the proof of part (4) above, we simply note that $\frac{1}{2}z \in D(\clq_\theta)^*$, but it is not invertible under composition. Indeed $\frac{1}{2}z$ has an inverse under composition, namely $2z$; however, since $2z$ is not a self-map of the unit disk, it does not belong to $D(\clq_\theta)$).

Classifications of the elements of $L(\clq_\theta)$ for Blaschke products considered below were obtained in \cite[Theorem 2.5]{Jav}. Here, however, we present a different classification, which will be useful. Let $ \text{Mob}(\mathbb{C}_\infty)$ denote the group of all M\"{o}bius transformations of $\mathbb{C}_\infty$.

\begin{theorem}
Let $\lambda \in \mathbb{D}$ be a nonzero scalar, $n \in \mathbb{N}$. Let
\[
\theta = z b_{\lambda}^n.
\]
Then $\vp \in L(\clq_\theta)^*$ if and only if $\varphi \in \mathrm{Mob}(\mathbb{C}_\infty)$ and
\[
\vp\left(\frac{1}{\bar{\lambda}}\right) = \frac{1}{\bar{\lambda}}.
\]
\end{theorem}
\begin{proof}
Suppose $\vp \in L(\clq_\theta)^*$. By the definition of $L(\mathcal{\clq}_\theta)$, we know that $\varphi$ is a rational function. On the other hand, since (recall the basis functions from \eqref{eqn: basis special}) $\dfrac{1}{(1-\bar{\lambda}z)^n}\in \clq_\theta$, there exist scalars $\{c_0, c_1, \ldots, c_n\}$ such that
\[
\dfrac{1}{(1-\bar{\lambda}z)^n}\circ \varphi=\dfrac{1}{(1-\bar{\lambda} \varphi)^n}=c_0+\sum_{k=1}^n
\dfrac{c_{k}}{(1-\bar{\lambda}z)^k}.
\]
As $\varphi$ is nonconstant, there exists $j\in \{1,\dots, n\}$ such that $c_j\neq 0$. The right-hand side of the above equation has exactly one pole at $z= 1/\bar{\lambda}$, of order at most $n$. This implies that $\dfrac{1}{1-\bar{\lambda} \varphi}$ has a simple pole at $z=1/\bar{\lambda}$. Therefore, there exist scalars $a$ and $b (\neq 0)$ such that
\[
\dfrac{1}{1-\bar{\lambda} \varphi}=a+\dfrac{b}{1-\bar{\lambda}z},
\]
which implies that $\varphi$ must be a M\"{o}bius transformation fixing the point $1/\bar{\lambda}$. Indeed, by substituting $z=1/\bar{\lambda}$ into the above identity, the right-hand side becomes infinite, and hence the left-hand side must also be infinite. This occurs only when $\varphi$ fixes  $1/\bar{\lambda}$. For the converse direction, let $\varphi \in \text{Mob}(\mathbb{C}_\infty)$ and assume that
\[
\varphi\Big(\dfrac{1}{\bar{\lambda}}\Big)=\dfrac{1}{\bar{\lambda}}.
\]
Then $1-\bar{\lambda}\varphi(z)$ has a single zero at $1/\bar{\lambda}$. There exist scalars $\alpha, \beta \in \mathbb{C}$ such that
\begin{equation*}
\dfrac{1}{1-\bar{\lambda} \varphi}=\dfrac{az+b}{1-\bar{\lambda}z}=\alpha+\dfrac{\beta}{1-\bar{\lambda}z}.
\end{equation*}
Thus
\begin{equation*}
\dfrac{1}{(1-\bar{\lambda} \varphi)^k} \in span\left\{ 1, \dfrac{1}{1-\bar{\lambda}z}, \dots, \dfrac{1}{(1-\bar{\lambda}z)^k}\right\} \subseteq \clq_\theta,
\end{equation*}
for all $k=1, \ldots, n$, and consequently $C_\varphi(\clq_\theta) \subseteq \clq_\theta$. This completes the proof of the theorem.
\end{proof}

Analogous result for $D(\clq_\theta)$ is also true with the same lines of proof. Thus, we give representations of $D(\clq_\theta)$ and $L(\clq_\theta)$ as follows:

\begin{corollary}\label{cor: z b lambda}
Let $\lambda \in \mathbb{D} \setminus \{0\}$ and let $n \in \mathbb{N}$. Let
\[
\theta = z b_{\lambda}^n.
\]
Then
\[
L(\clq_\theta)= \left\{\vp \in  \mathrm{Mob}(\mathbb{C}_\infty): \vp\left(\frac{1}{\bar{\lambda}}\right) = \frac{1}{\bar{\lambda}} \right\} \cup
\left(\mathbb{C} \setminus \left\{\frac{1}{\bar{\lambda}}\right\}\right),
\]
and
\[
D(\clq_\theta)= \left\{\vp \in \cls(\D) \cap  \mathrm{Mob}(\mathbb{C}_\infty): \vp\left(\frac{1}{\bar{\lambda}}\right) = \frac{1}{\bar{\lambda}} \right\} \cup \mathbb{D}.
\]
Moreover, in this case, $D(\clq_\theta)^* \cap  \mathrm{Aut}(\mathbb{D})$ is uncountable, and $L(\clq_\theta)^*$ is an infinite group that is not cyclic.
\end{corollary}
\begin{proof}
The first two identities follow from the previous theorem. Pick $\varphi \in D(\clq_\theta)^* \cap  \text{Aut}(\mathbb{D})$. In particular, there exist $\alpha \in \mathbb{R}$ and $a \in \mathbb{D}$ such that
\begin{equation}\label{eqn: new 1}
\varphi(z)=e^{i\alpha} b_a.
\end{equation}
The condition $\vp\left(\dfrac{1}{\bar{\lambda}}\right)=\dfrac{1}{\bar{\lambda}}$ is equivalent to the identity
\begin{equation}\label{eqn: new 2}
e^{i\alpha} \bar{\lambda}(1- a \bar{\lambda})=\bar{\lambda}-\bar{a}.
\end{equation}
From this, we deduce that $|\lambda|=\left|\dfrac{a-\lambda}{1-a\bar{\lambda}}\right|$, which implies
\[
1-|\lambda|^2=1-\left|\dfrac{a-\lambda}{1-a\bar{\lambda}}\right|^2.
\]
Since
\[
1-\left|\dfrac{a-\lambda}{1-a\bar{\lambda}}\right|^2=\dfrac{(1-|\lambda|^2)(1-|a|^2)}{|1-a\bar{\lambda}|^2},
\]
and $1-|\lambda|^2\neq 0$, the above identity simplifies to
\[
|1-a\bar{\lambda}|^2+|a|^2=1,
\]
equivalently,
\[
(1+|\lambda|^2)|a|^2-2\text{Re}(a\bar{\lambda})=0.
\]
This implies that $a$ lies on the circle $C(\lambda)$, centered at $\dfrac{\lambda}{1+|\lambda|^2}$ with radius $\dfrac{|\lambda|}{1+|\lambda|^2}$, that is, $a \in C(\lambda)$, where
\[
C(\lambda) = \left\{z \in \mathbb{C}: \left|z-\dfrac{\lambda}{1+|\lambda|^2}\right|^2=\left(\dfrac{|\lambda|}{1+|\lambda|^2}\right)^2\right\}.
\]
Since $0<|\lambda|<1$, it follows that
\[
\dfrac{|\lambda|}{1+|\lambda|^2}<\dfrac{1}{2},
\]
and hence
\[
C(\lambda) \subseteq \mathbb{D}.
\]
In particular, the set $D(\clq_\theta)^* \cap  \text{Aut}(\mathbb{D})$ is uncountable. Since $D(\clq_\theta)^* \subseteq L(\clq_\theta)^*$, it follows that $L(\clq_\theta)^*$ is also uncountable. If a M\"{o}bius map $\varphi$ fixes $1/\bar{\lambda}$, then its inverse $\varphi^{-1}$ also fixes $1/\bar{\lambda}$. Hence $\varphi^{-1} \in L(\clq_\theta)^*$, which implies that $L(\clq_\theta)^*$ forms a group with uncountably many elements. Therefore, it cannot be a cyclic group. This concludes the proof.
\end{proof}

By \eqref{eqn: new 1} and \eqref{eqn: new 2}, we have
\[
D(\clq_\theta)^* \cap  \text{Aut}(\mathbb{D})= \left\{\frac{1}{\bar{\lambda}}  \overline{b_a(\lambda)}b_a: a \in C(\lambda)\right\}.
\]
Note also that $D(\clq_\theta)^*$ is a group if and only if every element in $D(\clq_\theta)^*$ is an automorphism of $\mathbb{D}$, that is,
\[
D(\clq_\theta)^* \cap  \text{Aut}(\mathbb{D})=D(\clq_\theta)^*.
\]

Next, we proceed to the case of Blaschke products that have the origin as a zero of higher multiplicity, along with  one nonzero zero:
\[
\theta = z^m b_\lambda^n,
\]
where $m\geq 2$ and $n\geq 1$. An incomplete statement addressing this setting appeared in \cite[Theorem 2.9]{Jav}. In particular, it omitted the most nontrivial aspect of the result. In fact, the result differs substantially between the cases $m-1\neq n$ and $m-1 = n$, as outlined below. First, we consider the case of $m-1\neq n$.

\begin{theorem}\label{thm m geq 2 lambda}
Let $m\geq 2$, $n\geq 1$ be natural numbers and let $\lambda \in \D \setminus \{0\}$. Define
\[
\theta = z^m b_\lambda^n.
\]
If $m-1\neq n$, then
\[
L(\clq_\theta)= \left\{(1-\bar{\lambda}a)z+a: a\neq \frac{1}{\bar\lambda} \right\} \cup
\left(\mathbb{C} \setminus \left\{\frac{1}{\bar{\lambda}}\right\}\right),
\]
and
\[
D(\clq_\theta) =  \left\{z\right\} \cup \D.
\]
\end{theorem}
\begin{proof}
In view of \eqref{eqn: basis all} and \eqref{eqn: basis special}, we have
\[
\clq_\theta=\text{span}\left\{1, z,\dots, z^{m-1}, \dfrac{1}{1-\bar{\lambda}z}, \dots, \dfrac{1}{\left(1-\bar{\lambda}z\right)^{n }}\right\}.
\]
If $\varphi \in \cls(\D)$ is constant, then $C_\varphi$ maps $H^2(\mathbb{D})$ into the space of constant functions. Therefore, only subspaces that contain constants can be invariant under $C_\varphi$. On the other hand, If $\varphi(z) = z$, then the operator $C_\varphi$ is the identity operator, implying that $\clq_\theta$ is invariant under $C_\varphi$. Thus $D(\clq_\theta)$ contains $\mathbb{D}\cup \{z\}$. To achieve a similar inclusion of $L(\clq_\theta)$, we consider $\varphi(z) = (1-\bar{\lambda}a)z+a$, $a\neq 1/\bar{\lambda}$. Then
\[
\dfrac{1}{1-\bar{\lambda}\varphi(z)}= \frac{1}{1-\bar{\lambda}a} \frac{1}{1-\bar{\lambda}z}.
\]
A simple computation (along with the representation of $\clq_\theta$ as above) then shows that $\clq_\theta$ is invariant under $C_\varphi$, and hence $L(\clq_\theta)$ contains the functions appearing on the right-hand side of the identity in the statement.

\noindent Now we turn to prove the reverse inclusions. First we do this for $L(\clq_\theta)$. Pick $\varphi\in L(\clq_\theta)^*$. Then, $\varphi$ is a nonconstant rational function of the form
\[
\varphi= \frac{p}{q},
\]
where $p$ and $q$ are polynomials with no common factors. As $\dfrac{1}{1-\bar{\lambda}z} \in\clq_\theta$,
we have
\[
\dfrac{1}{1-\bar{\lambda}\varphi} =C_\varphi\left(\dfrac{1}{1-\bar{\lambda}z}\right) \in \clq_\theta,
\]
and hence there exist scalars $c_1, \dots, c_{m-1}$ and $d_1, \dots, d_n$,  not all zero, such that
\begin{equation}\label{cd}
\dfrac{1}{1-\bar{\lambda}\varphi} =c_0+c_1z+\dots+c_{m-1}z^{m-1}+\sum_{k=1}^n
\dfrac{d_{k}}{(1-\bar{\lambda}z)^k}.
\end{equation}
In view of $\varphi= \frac{p}{q}$, we rewrite this as
\begin{equation}\label{cddd}
\dfrac{q}{q-\bar{\lambda}p} =c_0+c_1z+\dots+c_{m-1}z^{m-1}+\sum_{k=1}^n
\dfrac{d_{k}}{(1-\bar{\lambda}z)^k}.
\end{equation}
We claim that there exists $j\in \{1,\dots, n\}$ such that
\[
d_j\neq 0.
\]
To see this, assume to the contrary that $d_j=0$ for all $j \in \{1, \ldots, n\}$. By \eqref{cddd}, we know that $\dfrac{q}{q-\bar{\lambda}p}$ is a polynomial: 
\[
\dfrac{q}{q-\bar{\lambda}p} \in \mathbb{C}[z].
\]
First, let us consider the case where $q$ is constant. Since $\varphi$ is nonconstant, $p$ must be a nonconstant polynomial. Therefore, $\dfrac{q}{q-\bar{\lambda}p}$ cannot be a polynomial, which is a contradiction. Now, let us consider the case where $q$ is a nonconstant polynomial. As $z\in \clq_\theta$, we have $\varphi= \frac{p}{q} \in \clq_\theta$. Since $q$ is nonconstant it must vanish only at $1/\bar{\lambda}$ and have no other zeros. Because $\dfrac{q}{q-\bar{\lambda}p}$ is a polynomial, $q-\bar{\lambda}p$ must be a factor of $q$. If $q-\bar{\lambda}p$ is nonconstant, then it must vanish at $1/\bar{\lambda}$, and consequently, $p$ must also vanish at $1/\bar{\lambda}$, which contradicts the fact that $p$ and $q$ have no common factors. Now let us consider the remaining case where $q-\bar{\lambda}p$ is constant. Suppose $m, n$ satisfy that $m-1>n$. The map $\varphi$ has a pole at $1/\bar{\lambda}$ with order at least one. Consequently,
\[
C_\varphi(z^{m-1})=\varphi^{m-1},
\]
has pole at $1/\bar{\lambda}$ with order at least $m-1$ which is strictly bigger than $n$. Therefore,
\[
C_\varphi(z^{m-1})\notin \clq_\theta,
\]
a contradiction. Next, suppose $m, n$ satisfy $m-1<n$. Since $q-\bar{\lambda}p$ is constant and $q$ is nonconstant, we have
\[
\frac{q}{q-\bar{\lambda}p} = \frac{1}{1-\bar{\lambda}\varphi},
\]
is a polynomial of degree at least one. This implies that
\[
C_\vp \Big(\dfrac{1}{(1-\bar{\lambda}z)^n}\Big) = \dfrac{1}{(1-\bar{\lambda} \varphi)^n}
\]
is a polynomial of degree at least $n$, which is strictly bigger than $m-1$. Therefore,
\[
C_\vp \Big(\dfrac{1}{(1-\bar{\lambda}z)^n}\Big)\notin\clq_\theta,
\]
a contradiction. This proves the claim that $d_j\neq 0$ for some $j\in \{1,\dots, n\}$. Consequently, $1/(1-\bar{\lambda} \varphi)$ must have a pole at $1/\bar{\lambda}$ with order at least one. Hence, $\dfrac{1}{(1-\bar{\lambda} \varphi)^n}$ must have a pole at $1/\bar{\lambda}$ of order at least $n$. As
\[
\frac{1}{(1-\bar{\lambda}z)^n} \in \clq_\theta,
\]
it can have a pole at $1/\bar{\lambda}$ of order at most $n$. Therefore, $\dfrac{1}{(1-\bar{\lambda} \varphi)^n}$ must have a pole at $1/\bar{\lambda}$ of order $n$. Consequently, $1/(1-\bar{\lambda} \varphi)$ must have a pole at $1/\bar{\lambda}$ of order exactly one. Then \eqref{cd} becomes
\begin{equation*}
\dfrac{1}{1-\bar{\lambda}\varphi} =c_0+c_1z+\dots+c_{m-1}z^{m-1}+ \dfrac{d_{1}}{1-\bar{\lambda}z},
\end{equation*}
where $d_1\neq 0$. Since the right-hand side of the above identity has a simple pole at $1/\bar{\lambda}$, it follows that $1-\bar{\lambda} \varphi$ must vanish at $1/\bar{\lambda}$, that is,

\[
\varphi\left(1/\bar{\lambda}\right)= 1/\bar{\lambda}.
\]
As $z \in \clq_\theta$, we have $\varphi \in \clq_\theta$. On the other hand, $\varphi(1/\bar{\lambda})\neq \infty$ implies that $\varphi$ is a polynomial of degree at most $m-1$. It follows that
\[
\varphi(z)=b_0+b_1z+\dots+b_{m-1}z^{m-1},
\]
for some constants $b_0, b_1, \dots, b_{m-1}$. If $\deg \varphi \geq 2$, then $1-\bar{\lambda} \varphi$ is also polynomial of degree at least $2$. However, this is not possible, since $\dfrac{1}{1-\bar{\lambda}\varphi}$ has only a simple pole. Thus, we must have
\[
\varphi(z)=b_0+b_1z,
\]
with $b_1\neq 0$, because $\varphi$ is nonconstant. Consequently, $1-\bar{\lambda}\varphi$ is polynomial of degree $1$ and vanishing at $1/\bar{\lambda}$. There exists a nonzero scalar $\gamma$ such that
\begin{equation}\label{ld}
1-\bar{\lambda}\varphi(z)=\gamma(1-\bar{\lambda}z).
\end{equation}
This implies
\[
\varphi=\gamma z+\left(\frac{1-\gamma}{\bar{\lambda}}\right),
\]
equivalently,
\[
\varphi=(1-\bar{\lambda}a)z+a,
\]
for some $\gamma\neq 0$ and $a\neq 1/\bar{\lambda}$. This establishes the other inclusion of the set $L(\clq_\theta)$ as described in the statement. For the corresponding inclusion of the set $D(\clq_\theta)$, from \eqref{ld} we obtain
\[
\dfrac{1}{1-\bar{\lambda}\varphi(z)}=\dfrac{\frac{1}{\gamma}}{1-\bar{\lambda}z}.
\]
By Lemma \ref{lem1}, it follows that $\varphi(z)=z$, $z \in \mathbb{D}$. This completes the proof.
\end{proof}

Now we turn to the case $m-1= n$. As we shall see, the situation differs significantly from the case $m-1 \neq n$.

\begin{theorem}\label{thm m geq 2 lambda-2}
Let $m\geq 2$, $n\geq 1$ be natural numbers and let $\lambda \in \D \setminus \{0\}$. Suppose $m-1= n$, and define
\[
\theta = z^m b_\lambda^n=z^{n+1} b_\lambda^n.
\]
Then
\[
L(\clq_\theta)= \left\{(1-\bar{\lambda}a)z+a:a\neq \frac{1}{\bar{\lambda}} \right\} \cup
\left\{\frac{c-z}{1-\bar{\lambda}z}:c\neq \frac{1}{\bar{\lambda}} \right\} \cup
 \left(\mathbb{C} \setminus \left\{\frac{1}{\bar{\lambda}}\right\}\right),
\]
and
\[
D(\clq_\theta) =L(\clq_\theta)\cap \cls(\D) =  \left\{z\right\} \cup \left\{\frac{\lambda-z}{1-\bar{\lambda}z} \right\} \cup \D.
\]
\end{theorem}
\begin{proof}
We begin by recalling as in the proof of the previous theorem that
\[
\clq_\theta=\text{span}\left\{1, z,\dots, z^{n}, \dfrac{1}{1-\bar{\lambda}z}, \dots, \dfrac{1}{\left(1-\bar{\lambda}z\right)^{n }}\right\}.
\]
By computations similar to those in the proof of the previous theorem, it is easy to see that $L(\clq_\theta)$ contains the functions on the right-hand side given in the statement. To prove the reverse inclusion, assume that $\varphi\in L(\clq_\theta)^*$. Then $\varphi$ is a nonconstant rational function of the form
\[
\varphi= \frac{p}{q},
\]
where $p$ and $q$ are polynomials with no common factors. We prove the result by considering three distinct cases:

\noindent \textsf{Case I. $p$ is constant}: In this case, $q$ must be nonconstant since $\varphi$ is nonconstant. As $\varphi\in\clq_\theta$ and is not a polynomial, it follows that $\vp$ must have a pole at $1/\bar{\lambda}$. Moreover, $\vp^n\in\clq_\theta$ cannot have a pole at $1/\bar{\lambda}$ order strictly greater than $n$, which implies that $1/\bar{\lambda}$ is a simple pole of $\vp$. Since $p$ is constant, we have
\[
\vp(z)=\frac{c}{1-\bar{\lambda}z},
\]
for some constant $c\neq 0$. Consequently, $\dfrac{1}{1-\bar{\lambda}\varphi}$ is not a polynomial and therefore must have a pole at $1/\bar{\lambda}$. This would require
\[
\vp\Big(\frac{1}{\bar{\lambda}}\Big) = \frac{1}{\bar{\lambda}},
\]
which is impossible since $\vp(1/\bar{\lambda})=\infty$.

\noindent \textsf{Case II. $q$ is constant}: In this case, $\vp$ is a nonconstant polynomial. Since $\vp^n\in\clq_\theta$ cannot be a polynomial of degree strictly greater than $n$, it follows that $\vp$ is an affine map (a polynomial of degree one). Hence, $\dfrac{1}{1-\bar{\lambda}\vp}$ is not a polynomial. Moreover, as
\[
C_\vp \Big(\dfrac{1}{(1-\bar{\lambda}z)^n}\Big) = \dfrac{1}{(1-\bar{\lambda} \varphi)^n}\in\clq_\theta,
\]
we conclude that $\dfrac{1}{1-\bar{\lambda}\vp}$ has simple pole at $1/\bar{\lambda}$. We also have that $1-\bar{\lambda}\vp$ is a first-degree polynomial, we must have
\[
1-\bar{\lambda}\vp=c (1-\bar{\lambda}z)
\]
for some constant $c\neq 0$. As in previous Theorem \ref{thm m geq 2 lambda}, this yields
\[
\varphi=(1-\bar{\lambda}a)z+a,\, a\neq \frac{1}{\bar{\lambda}}.
\]

\noindent \textsf{Case III. Neither $p$ nor $q$ is constant}: Since $\vp$ is not a polynomial, as before, we deduce that $\vp$ has a simple pole at $1/\bar{\lambda}$, equivalently, $q$ is a first-degree polynomial vanishing at $1/\bar{\lambda}$. Thus,
\[
\infty=\vp\Big(\frac{1}{\bar{\lambda}}\Big)\neq \frac{1}{\bar{\lambda}},
\]
which implies that $1-\bar{\lambda}\vp$ does not vanish at $1/\bar{\lambda}$. Then $1/\bar{\lambda}$ cannot be a pole of $\dfrac{1}{1-\bar{\lambda}\vp}$, and hence $\dfrac{1}{1-\bar{\lambda}\vp}$ is a polynomial. Since $\dfrac{1}{(1-\bar{\lambda}\vp)^n} \in\clq_\theta$, we have
\[
\frac{q}{q-\bar{\lambda}p}=\frac{1}{1-\bar{\lambda}\vp},
\]
a polynomial of degree one. As $q$ is a first-degree polynomial vanishing at $1/\bar{\lambda}$, we get $q-\bar{\lambda}p$ is a nonzero constant, and hence
\[
\frac{1}{1-\bar{\lambda}\vp}=\frac{1}{\gamma} (1-\bar{\lambda}z),
\]
for some nonzero constant $\gamma$. Equivalently,
\[
\vp(z)=\frac{\frac{1-\gamma}{\bar{\lambda}}-z}{1-\bar{\lambda}z}=\frac{c-z}{1-\bar{\lambda}z},
\]
for some constant $c\neq 1/\bar{\lambda}$. Hence verifies the statement for $L(\clq_\theta)$. For $D(\clq_\theta)$, observe that $(1-\bar{\lambda}a)z+a$ is self-map of $\D$ if and only if it is the identity. Next, we determine all $c\neq 1/\bar{\lambda}$ such that
\[
z \mapsto \dfrac{c-z}{1-\bar{\lambda}z},
\]
is self-map of $\D$. This holds if and only if (cf. \cite{Mut})
\begin{equation}\label{D}
|c-\lambda|+|1-c\bar{\lambda}|\leq 1-|\lambda|^2.
\end{equation}
Set $w=c-\lambda$. Then
\[
c\bar{\lambda}=(w+\lambda)\bar{\lambda}=|\lambda|^2+w\bar{\lambda}.
\]
Thus,
\[
|1-c\bar{\lambda}|=|1-|\lambda|^2 - w\bar{\lambda}|\geq 1-|\lambda|^2-|w\bar{\lambda}|.
\]
Consequently,
\[
|w|+|1-c\bar{\lambda}|\geq 1-|\lambda|^2+ |w|(1-|\lambda|).
\]
Now, \eqref{D} implies
\[
1-|\lambda|^2+ |w|(1-|\lambda|)\leq 1-|\lambda|^2,
\]
which forces $|w|=0$. Hence, $c=\lambda$. This completes the proof.
\end{proof}

The following is now immediate:

\begin{corollary}\label{cor: m 2 n 1 lambda}
Let $m\geq 2$ and $n\geq 1$ be natural numbers, and let $\lambda \in \D \setminus \{0\}$. Define
\[
\theta = z^m b_\lambda^n.
\]
Then the following hold:
\begin{enumerate}
\item $L(\clq_\theta)^*$ is an uncountable group.
\item If $m-1\neq n$, then $D(\clq_\theta)^*$ is a trivial group.
\item If $m-1=n$, then $D(\clq_\theta)^*$ is a cyclic group of order $2$.
\end{enumerate}
\end{corollary}

With this result, we now have a clearer picture of the structure of the sets $D(\clq_\theta)$ and $L(\clq_\theta)$ when $\theta = z^m b_\lambda^n$ for $\lambda \in \D \setminus \{0\}$ and $m, n \geq 1$.

\section{Blaschke products vanishing at $0$ and $\# \clz(\theta) \geq 3$}\label{sec: mob}

As in the previous section, we continue here under the assumption that the finite Blaschke product $\theta$ vanishes at the origin. Here, we assume that $\theta$ has at least two distinct zeros, aside from the origin:
\[
\# (\clz(\theta)\setminus \{0\}) \geq 2.
\]
The main difference in this setting is that the sets $D(\clq_\theta)$ and $L(\clq_\theta)$ will contain more M\"{o}bius transformations.

Given a M\"{o}bius transformation $\varphi(z)=\dfrac{az+b}{cz+d}$, we construct another M\"{o}bius transformation $\tilde{\vp}$ as
\[
\tilde{\varphi}(z)=\dfrac{\bar{a}z-\overline{c}}{-\bar{b}z+\bar{d}}.
\]
Recall, for a self-map $f$ and a natural number $n$, we write
\[
f^{[n]}= \underbrace{f \circ \dots \circ f}_{n \text{ times }}.
\]
These transformations behave naturally:

\begin{lemma}\label{pro1}
Let $\vp$ be a M\"{o}bius transformation, and let $n \in \mathbb{N}$. We have the following:
\begin{itemize}
\item[(i)]  $\varphi \in \cls(\D)$ if and only if $\tilde{\varphi} \in \cls(\D)$.
\item[(ii)] $\varphi^{[n]}=z$ if and only if $\tilde{\varphi}^{[n]}=z$.
\end{itemize}
\end{lemma}
\begin{proof}
Define $\eta(z)=\overline{\varphi(\bar{z})}=\dfrac{\bar{a}z+\overline{b}}{\bar{c}z+\bar{d}}$. Then
\[
\eta=\bar{z}\circ \varphi \circ \bar{z}.
\]
Given that $\eta^{-1}(z)=\dfrac{\bar{d}z-\bar{b}}{-\bar{c}z+\bar{a}}$, we have
\[
\tilde{\varphi}=\dfrac{1}{z}\circ \eta^{-1}\circ \dfrac{1}{z}.
\]
In view of this, we see that $\varphi$ is a self-map of $\mathbb{D}$ if and only if $\eta$ is a self-map of $\mathbb{D}$ if and only if $\eta^{-1}$ is a self-map of
\[
\{z: |z|>1\}\cup \{\infty\},
\]
if and only if $\tilde{\varphi}$ is a self-map of $\mathbb{D}$. This completes the proof of part (i). For (ii), we proceed similarly:
\[
\varphi^{[n]}=z \Leftrightarrow \eta^{[n]}=z \Leftrightarrow (\eta^{-1})^{[n]}=z \Leftrightarrow \tilde{\varphi}^{[n]}=z.
\]
This completes the proof of the lemma.
\end{proof}

In the following, we consider finite Blaschke products $\theta$ satisfying $\theta(0)=0$, $\theta'(0)\neq 0$, and $\# (\clz(\theta)\setminus \{0\}) \geq 2$. These conditions imply that there exists a finite Blaschke product $\hat{\theta}$, not vanishing at the origin, such that $\theta = z \hat{\theta}$ and $\# (\clz(\hat\theta)) \geq 2$. The following is a part of \cite[Lemma 2.7 and Lemma 2.10]{Jav}.

\begin{lemma}\label{propo}
Let $\theta$ be a finite Blaschke product satisfying $\theta(0)=0$ and $\# (\clz(\theta)\setminus \{0\}) \geq 2$. If $\varphi \in L(\clq_\theta)^*$, then $\varphi$ is a M\"obius  transformation:
\[
\varphi(z)=\dfrac{az+b}{cz+d},
\]
for some complex numbers $a, b, c, d$ satisfying $ad-bc\neq 0$.
\end{lemma}

Note that the full-length version of Lemma 2.7 in \cite{Jav} concerns representations of functions $\varphi \in L(\clq_\theta)$  for the class of finite Blaschke products $\theta$ described above. We now use the above part of the result to provide an alternative classification of functions in the sets $D(\clq_\theta)^*$ and $L(\clq_\theta)^*$:

\begin{theorem}\label{lem2.7}
Let $\theta$ be a finite Blaschke product satisfying $\theta(0)=0$, $\theta'(0)\neq 0$, and $\# (\clz(\theta)\setminus \{0\}) \geq 2$. Then $\varphi \in L(\clq_\theta)^*$  if and only if $\varphi$ is a M\"obius transformation and
\[
mult_\theta(\lambda)=mult_\theta(\tilde{\varphi}(\lambda)),
\]
for all $\lambda \in \clz(\theta)\setminus \{0\}$. Moreover, we have
\[
D(\clq_\theta)^*=L(\clq_\theta)^* \cap \mathcal{S}(\mathbb{D}).
\]
\end{theorem}
\begin{proof}
Suppose $\varphi \in L(\clq_\theta)^*$. By Lemma \ref{propo}, $\varphi$ is a M\"obius transformation. Let
\[
\varphi(z)=\dfrac{az+b}{cz+d}.
\]
Fix $\lambda \in \clz(\theta)\setminus \{0\}$, and let $m = mult_\theta (\lambda)$. Now
\[
\dfrac{1}{1-\bar{\lambda}z} \circ \varphi = \dfrac{1}{1-\bar{\lambda}\varphi}
=\dfrac{1}{1-\bar{\lambda}\dfrac{az+b}{cz+d}}
= \dfrac{\alpha z+ \beta}{1-\overline{\dfrac{\overline{a}\lambda-\overline{c}}{-\overline{b} \lambda+\overline{d}} }\cdot z} \in \clq_\theta,
\]
for some scalars $\alpha$ and $\beta$. Then
\begin{equation}\label{rep}
\dfrac{1}{1-\bar{\lambda}\varphi}= \dfrac{\alpha z+ \beta}{1-\overline{\tilde{\varphi}(\lambda)}z}=\gamma+  \dfrac{\delta}{1-\overline{\tilde{\varphi}(\lambda)}z} \in \clq_\theta,
\end{equation}
for some scalars $\delta \neq 0, \gamma \in \mathbb{C}$ (note that if $\delta=0$, then by the above identity, we get $\varphi$ is constant, which is not possible). This shows (in view of the basis functions in $\clq_\theta$ as in \eqref{eqn: basis all}) that $\tilde{\varphi} \in \clz(\theta)$. Moreover, for each $k \in \{1, \dots, m\}$, we have
\[
\dfrac{1}{(1-\bar{\lambda}\varphi)^k} = \dfrac{1}{(1-\bar{\lambda}z)^k}\circ \varphi \in \clq_\theta,
\]
and hence
\begin{equation*}
\left(  \dfrac{\delta}{1-\overline{\tilde{\varphi}(\lambda)}z}\right)^k= \left(\dfrac{1}{1-\bar{\lambda}\varphi}-\gamma\right)^k= \sum_{j=0}^k \dfrac{c_j}{(1-\bar{\lambda}\vp)^j} \in \clq_\theta
\end{equation*}
for some scalars $c_1, \ldots, c_k$. Therefore, $\tilde{\varphi}(\lambda) \in \clz(\theta)$ and
\[
m=mult_\theta(\lambda)\leq mult_\theta (\tilde{\varphi}(\lambda)).
\]
If we set $\tilde{\varphi}^{[0]}:=z$, then, by induction, we have
\begin{equation}\label{eqn:mul}
mult_\theta(\tilde{\varphi}^{[k]}(\lambda)) \leq mult_\theta (\tilde{\varphi}^{[k+1]}(\lambda)),
\end{equation}
for all $k \in \mathbb{Z}_+$. This implies
\[
\tilde{\varphi}^{[k+1]}(\lambda) \in \clz(\theta),
\]
for all $k \in \mathbb{N}$. Since $\theta$ has only finitely many zeros, it follows that $\tilde{\varphi}^{[s]}(\lambda)=\tilde{\varphi}^{[t]}(\lambda)$ for some $s, t \in \mathbb{N}$ with $s<t$. Applying $\tilde{\varphi}^{[-s]}$ to the both sides yields
\[
\tilde{\varphi}^{[t-s]}(\lambda)=\lambda.
\]
Therefore, by \eqref{eqn:mul}, we have
\[
mult_\theta(\lambda)=mult_\theta (\tilde{\varphi}(\lambda)).
\]
Conversely, suppose $\varphi$ is M\"{o}bius transformation of the form $\varphi(z)=\dfrac{az+b}{cz+d}$ with $ad-bc\neq 0$, and assume that $mult_\theta(\lambda)=mult_\theta (\tilde{\varphi}(\lambda))$ for all $\lambda \in \clz(\theta)\setminus \{0\}$. Fix $\lambda \in \clz(\theta)\setminus \{0\}$, and let $m=mult_\theta(\lambda)$. As before,
\begin{equation*}
\dfrac{1}{1-\bar{\lambda}\varphi} = \gamma+ \dfrac{\delta}{1-\overline{\tilde{\varphi}(\lambda)}z},
\end{equation*}
for some scalars $\delta \neq 0$ and $\gamma \in \mathbb{C}$. Since $mult_\theta (\tilde{\varphi}(\lambda))=m$, we have
\begin{equation*}
\dfrac{1}{(1-\overline{\tilde{\varphi}(\lambda)}z)^j} \in \clq_\theta,
\end{equation*}
for all $j=1, \ldots, m$. As $1 \in \clq_\theta$, we also have
\begin{equation*}
\dfrac{1}{(1-\bar{\lambda}z)^k} \circ \varphi=\dfrac{1}{(1-\bar{\lambda}\varphi)^k}=\left(\gamma +\dfrac{\delta}{1-\overline{\tilde{\varphi}(\lambda)}z}\right)^k \in \clq_\theta,
\end{equation*}
for all $k \in \{1, \ldots, m\}$. Thus, $C_\varphi$ maps every basis element of $\clq_\theta$ into $\clq_\theta$, and hence $\varphi \in L(\clq_\theta)^*$. The identity $D(\clq_\theta)^*=L(\clq_\theta)^* \cap \mathcal{S}(\mathbb{D})$  follows trivially. This completes the proof.
\end{proof}

For each $\tilde \varphi$, where $\varphi$ is in the set $L(\clq_\theta)^*$, we have the following result:

\begin{corollary}\label{per}
Let $\theta$ be a finite Blaschke product satisfying $\theta(0)=0$, $\theta'(0)\neq 0$, and $\# (\clz(\theta)\setminus \{0\}) \geq 2$. If $\varphi \in L(\clq_\theta)^*$, then $\tilde{\varphi}$ is a permutation on
$\clz(\theta)\setminus \{0\}$.
\end{corollary}
\begin{proof}
Fix $\lambda \in A:=\clz(\theta) \setminus \{0\}$. There exist scalars $\delta \neq 0$ and $\gamma \in \mathbb{C}$ such that
\begin{equation*}
\dfrac{1}{1-\bar{\lambda}\varphi}=\gamma+  \dfrac{\delta}{1-\overline{\tilde{\varphi}(\lambda)}z}
\end{equation*}
and $\tilde{\varphi}(\lambda) \in \clz(\theta)$. If $\tilde{\varphi}(\lambda)=0$, then this would imply $\varphi$ is constant, which is not possible. Thus, $\tilde{\varphi}(\lambda) \in A$. Since $\tilde{\varphi}$ is M\"obius (and in particular, injective) map, it defines a permutation on $A$.
\end{proof}

Theorem \ref{lem2.7} can be equivalently reformulated as the following result.

\begin{corollary}\label{inv}
Let $\theta$ be a finite Blaschke product satisfying $\theta(0)=0$, $\theta'(0)\neq 0$, and $\# (\clz(\theta)\setminus \{0\}) \geq 2$. Then $\varphi \in L(\clq_\theta)^*$  if and only if $\varphi$ is a M\"obius transformation and
\[
mult_\theta(\lambda)=mult_\theta(\tilde{\varphi}^{-1}(\lambda))=mult_\theta\left(\frac{\overline{c}+\overline{d}\lambda}{\overline{a}+\overline{b}\lambda} \right),
\]
for all $\lambda \in \clz(\theta)\setminus \{0\}$.
\end{corollary}

We now present a simpler and substantially improved version of Lemma 2.10 in \cite{Jav}.

\begin{theorem}\label{bm}
Let $\theta$ be a finite Blaschke product. Suppose    $\#(\clz(\theta)\setminus \{0\}) \geq 2$ and $mult_\theta(0)\geq 2$. Then $\varphi \in L(\clq_\theta)^*$ if and only if $\varphi$ is an affine transformation
\begin{equation*}
\varphi(z)=\alpha z+\beta,
\end{equation*}
for some scalars $\alpha(\neq 0)$ and $\beta$ such that
\[
mult_\theta(\lambda)=mult_\theta(\tilde{\varphi}(\lambda)) \mbox{~ for all~ } \lambda \in \clz(\theta)\setminus \{0\}.
\]
\end{theorem}
\begin{proof}
Fix $\varphi \in L(\clq_\theta)^*$. By Lemma \ref{propo}, Corollary \ref{per}, and an argument similar to the proof of Theorem \ref{lem2.7}, we obtain that
\[
\varphi=\frac{az+b}{cz+d},
\]
that is, $\varphi$ is a M\"obius transformation. Moreover, $\tilde{\varphi}$ is a permutation on
$\clz(\theta)\setminus \{0\}$, and
\[
mult_\theta(\lambda)=mult_\theta(\tilde{\varphi}(\lambda)),
\]
for all $\lambda \in \clz(\theta)\setminus \{0\}$. As $z\in \clq_\theta$, we  have $\varphi \in \clq_\theta$, that is,
\[
\varphi=\frac{az+b}{cz+d}= \gamma+\frac{\delta}{1-\overline{\left(
-\bar{c}/\bar{d}\,\right)} z} \in \clq_\theta.
\]
This implies $\tilde{\varphi}(0)=-\bar{c}/\bar{d}\in \clz(\theta)$. Since $\tilde{\varphi}$ is injective map on $\mathbb{C}_\infty$ and a permutation on $\clz(\theta)\setminus \{0\}$, it follows that $\tilde{\varphi}(0)=0$. Hence $c=0$, and therefore $\varphi$ is an affine map. This completes the proof of the forward implication. For the converse, the verification is straightforward: for each $\lambda \in \clz(\theta)\setminus \{0\}$, there exist a constant $\delta$ such that
\begin{equation*}
\dfrac{1}{1-\bar{\lambda}\varphi}= \dfrac{\delta}{1-\overline{\tilde{\varphi}(\lambda)}z},
\end{equation*}
which completes the proof of the theorem.
\end{proof}

Since $\tilde{\varphi}$ is a permutation on $\clz(\theta)\setminus \{0\}$, Theorem \ref{bm} can be reformulated in the following equivalent form:

\begin{corollary}\label{inv1}
Let $\theta$ be a finite Blaschke product satisfying $\theta(0)=0$, $mult_\theta(0)\geq 2$, and let $\# (\clz(\theta)\setminus \{0\}) \geq 2$. Then $\varphi \in L(\clq_\theta)^*$ if and only if $\varphi(z) = az+b$ with
\[
mult_\theta(\lambda)=mult_\theta(\tilde{\varphi}^{-1}(\lambda))=mult_\theta\left(\frac{\lambda}{\overline{a}+\overline{b}\lambda} \right),
\]
for all $\lambda \in \clz(\theta)\setminus \{0\}$.
\end{corollary}

The elements of $D(\clq_\theta)^*$ are special disc automorphisms of the unit disc:

\begin{corollary}\label{lary1}
Let $\theta$ be a finite Blaschke product satisfying $\theta(0)=0$, and $\# (\clz(\theta)\setminus \{0\}) \geq 2$, and let $\varphi \in D(\clq_\theta)^*$. Then $\varphi$ is a rational elliptic automorphism of $\D$. That is, there exists $m \in \mathbb{N}$ such that  $\varphi^{[m]}=z$.
\end{corollary}
\begin{proof}
Pick $\varphi \in D(\clq_\theta)^*$. As $D(\clq_\theta)^*\subseteq L(\clq_\theta)^*$, Corollary \ref{per} implies that $\tilde{\varphi}$ is a permutation on $A=\clz(\theta)\setminus \{0\}$. That is, $\tilde{\varphi}\in S_N$, where
\[
N:=\#A.
\]
There exist $m\in \mathbb{N}$ (for instance, $m=N!$) such that
\[
\tilde{\varphi}^{[m]}=z \mbox{~ on ~} A.
\]
In other words, every point of $A$ is a fixed point of $\tilde{\varphi}^{[m]}$. As a consequence of the Schwarz lemma, we conclude that
\[
\tilde{\varphi}^{[m]}=z \mbox{~ on ~} \mathbb{D},
\]
since a holomorphic self-map of $\D$ with at least two fixed points must be the identity. By applying part (ii) of Proposition \ref{pro1}, we obtain that $\varphi^{[m]}=z$. Thus,
\begin{equation*}
\varphi \circ \varphi^{[m-1]}= \varphi^{[m-1]} \circ \varphi=z.
\end{equation*}
That is,  $\varphi :\mathbb{D} \to \mathbb{D}$ is bijective. Hence $\varphi$ is an automorphism of $\mathbb{D}$.
\end{proof}

We can now also draw conclusions about the group structures of the sets $L(\clq_\theta)^*$ and $D(\clq_\theta)^*$:

\begin{corollary}\label{cor: general group}
Let $\theta$ be a finite Blaschke product satisfying $\theta(0)=0$, and $\# (\clz(\theta)\setminus \{0\}) \geq 2$. Then $L(\clq_\theta)^*$ and $D(\clq_\theta)^*$ are groups under composition.
\end{corollary}
\begin{proof}
Let $\varphi \in L(\clq_\theta)^*$. By Corollaries \ref{inv} and \ref{inv1}, we have
\[
mult_\theta(\lambda)=mult_\theta(\tilde{\varphi}^{-1}(\lambda)),
\]
for all $\lambda \in \clz(\theta)\setminus \{0\}$, which can be rewritten as
\[
mult_\theta(\lambda)=mult_\theta(\widetilde{\varphi^{-1}}(\lambda)),
\]
for all $\lambda \in \clz(\theta)\setminus \{0\}$. By Theorems \ref{lem2.7} and \ref{bm}, this yields $\varphi^{-1} \in L(\clq_\theta)^*$, and consequently, $L(\clq_\theta)^*$ is a group under composition. For the group structure of $D(\clq_\theta)^*$, pick $\varphi \in D(\clq_\theta)^*$. Then $\varphi^{[k]} \in D(\clq_\theta)^*$ for all $k \in \mathbb{N}$. By Corollary \ref{lary1}, there exists $n \in \mathbb{N}$ such that
\[
\varphi^{-1}= \varphi^{[n-1]}\in D(\clq_\theta)^*,
\]
which implies that $D(\clq_\theta)^*$ is a group under composition.
\end{proof}

In the latter part of the following section, we will discuss concrete examples of finite Blaschke products that satisfy the conditions of the above corollary. In fact, we will focus on those finite Blaschke products $\theta$ that additionally satisfy the condition
\[
\# (\clz(\theta)\setminus \{0\}) = 2.
\]

\section{Examples}\label{sec: example}

This section presents some concrete results in the form of examples. We provide two sets of example-based results, both stemming from two incorrect claims in \cite{Jav}.

The first focus is an example that was presented as an illustration of Theorem 2.4 in \cite{Jav}, and which appeared immediately after Corollary 2.4 in the same reference. In this section, we point out that the conclusion drawn in that example is not necessarily valid in all cases. We also correct the error and provide a precise statement that accurately reflects the situation illustrated by the example in \cite{Jav}.

We recall the setting of \cite{Jav}: Fix distinct points $\{\lambda_1, \ldots, \lambda_s\}$ from $\mathbb{D}\setminus \{0\}$ and fix a natural number $n (\geq 2)$. Set
\[
a=e^{\frac{2\pi i}{n}},
\]
and define
\[
\theta = \prod_{j=1}^s \left(\prod_{k=1}^nb_{{a^k\lambda_j}}\right)^{m_j}.
\]
In this case, the statement following Corollary 2.4 in \cite{Jav} asserts that
\[
D(\clq_\theta)=\{z, az, \dots, a^{n-1}z\}.
\]
We first point out that this conclusion is not correct: Consider the case with $n=2$, and
\[
a=-1.
\]
Choose
\[
s=2, \lambda_1= \frac{1}{2}, \lambda_2= \frac{i}{2},
\]
along with $m_1=m_2=1$. Then
\[
\theta = b_{-\frac{1}{2}} b_{\frac{1}{2}} b_{-\frac{i}{2}} b_{\frac{i}{2}},
\]
and according to \cite{Jav}, one concludes that
\begin{equation}\label{eqn: Javad wrong}
D(\clq_\theta)=\{z, -z\}.
\end{equation}
We apply the results proved in this paper to show that this conclusion is incorrect. Note that
\[
\clz(\theta)= \left\{\pm \frac{1}{2}, \pm \frac{i}{2}\right\},
\]
and all zeros have multiplicity one. In particular, we have
\[
mult_\theta (\lambda)=mult_\theta(i\lambda),
\]
for all $\lambda \in \clz(\theta)$. By part (2) of Corollary \ref{coro1}, we have $iz \in D(\clq_\theta)$, and hence
\[
\langle iz \rangle =\{\pm z, \pm iz\} \subseteq D(\clq_\theta).
\]
On the other hand, by Theorem \ref{thm3}, we have $D(\clq_\theta) \subseteq \langle iz \rangle$, and consequently
\[
D(\clq_\theta)=\{\pm z, \pm iz\}.
\]
This is clearly different from the identity claimed in \eqref{eqn: Javad wrong}. In this case, observe that all the zeros are simple. Hence, even if one assumes that the zeros of $\theta$ are all distinct, claim \eqref{eqn: Javad wrong} is false. However, there is a fix to this claim, which we state as follows:

\begin{theorem}
Let $n\geq 2$ be a natural number, and let $\{\lambda_1, \ldots, \lambda_s\} \subseteq \mathbb{D}\setminus \{0\}$. Set
\[
\theta = \prod_{j=1}^s \left(\prod_{k=1}^nb_{{a^k\lambda_j}}\right)^{m_j},
\]
where
\[
a=e^{\frac{2\pi i}{n}}.
\]
Assume that $\{|\lambda_1|, \ldots, |\lambda_s|\}$ is a set of distinct real numbers. Then
\[
D(\clq_\theta)=\langle az \rangle.
\]
\end{theorem}
\begin{proof}
By definition of $\theta$, we have
\[
\clz(\theta)=\bigcup_{j=1}^s \left\{\lambda_j, a\lambda_j, a^2 \lambda_j, \dots, a^{n-1}\lambda_j\right\},
\]
with the property that
\[
mult_\theta(a^k \lambda_j)=m_j,
\]
for all $k=1, \ldots, n$. Therefore, $mult_\theta(\lambda)=mult_\theta(a\lambda)$ for all $\lambda \in \clz(\theta)$. Thus, by part (2) of  Corollary \ref{coro1}, this implies that $az \in D(\clq_\theta)$, and hence
\[
\langle az \rangle \subseteq D(\clq_\theta),
\]
as $D(\clq_\theta)$ is a group. To prove the reverse inclusion, we pick $\varphi \in D(\clq_\theta)$. We know by Theorem \ref{thm2} that
\[
\varphi = \omega z,
\]
for some $\omega \in \T$. Since $\lambda_1 \in\clz(\theta)$, again, by part (2) of Corollary \ref{coro1}, $\omega \lambda_1 \in \clz(\theta)$ and $|\omega \lambda_1|=|\lambda_1|$. As $|\lambda_p| \neq |\lambda_q|$ for all $p \neq q$, it follows that
\[
\omega \lambda_1 \in \left\{\lambda_1, a\lambda_1, \dots, a^{n-1}\lambda_1\right\},
\]
and hence, there exists $k \in \left\{0, 1, \ldots, n-1\right\}$ such that
\[
\omega=a^k.
\]
This proves that $D(\clq_\theta) \subseteq \langle az \rangle$, and consequently, we have $D(\clq_\theta) = \langle az \rangle$.
\end{proof}

The example at the beginning of this section shows that the additional assumption $|\lambda_i| \neq |\lambda_j|$ for all $i \neq j$ (as opposed to $\lambda_i \neq \lambda_j$ for all $i \neq j$) cannot be omitted from the theorem. There are other ways to obtain a similar conclusion. For instance, we have the following: Let $n\geq 2$ be a natural number, and let $\{\lambda_1, \ldots, \lambda_s\} \subseteq \mathbb{D}\setminus \{0\}$ be a set of distinct scalars. Set
\[
\theta = \prod_{j=1}^s \left(\prod_{k=1}^nb_{{a^k\lambda_j}}\right)^{m_j},
\]
where $a=e^{\frac{2\pi i}{n}}$. Assume that $m_i \neq m_j$ for all $i \neq j$. Then
\[
D(\clq_\theta)=\langle az \rangle.
\]
The proof of $\langle az \rangle \subseteq D(\clq_\theta)$ is the same. As above, pick $\varphi = \omega z \in D(\clq_\theta)$ for some $\omega \in \T$. As $\lambda_1 \in \clz(\theta)$ and $\omega z \in D(\clq_\theta)$, it follows from part (2) of Corollary \ref{coro1} that $\omega \lambda_1 \in \left\{\lambda_1, a\lambda_1, \dots, a^{n-1}\lambda_1\right\}$. The remainder of the proof proceeds as before.

Once again, our example shows that the additional condition on multiplicities cannot, in general, be omitted unless the $|\lambda_j|$'s are all distinct.

We now turn to the claim made in \cite[Theorem 2.8]{Jav}. We note that this claim is incorrect. To that end, we consider the simplest example of a finite Blaschke product $\theta$ satisfying the conditions $\theta(0)=0$, $\theta'(0)\neq 0$, and
\[
\# (\clz(\theta)\setminus \{0\}) = 2.
\]
The following result, in particular, highlights a contrast between the sets $L(\clq_\theta)^*$ and $D(\clq_\theta)^*$.

\begin{theorem}\label{thm: 2 zero L uncount}
Let $\alpha$ and $\beta$ be two distinct nonzero elements of $\mathbb{D}$ and $m,n\in \mathbb{N}$. Let $\tilde{b}_{\alpha,\beta}$ denote the unique disc automorphism that interchanges $\alpha$
and $\beta$. Define
\[
\theta=zb_\alpha^m b_\beta^n.
\]
Then the following two properties hold:
\begin{enumerate}
\item $L(\clq_\theta)^*$ is an uncountable group.
\item  $D(\clq_\theta)^* =
\begin{cases}
\{z\} & \text{ if } m\neq n \\
\{z, b_{\alpha,\beta}\} & \text{otherwise.}
\end{cases}$
\end{enumerate}
\end{theorem}
\begin{proof}
Fix $\gamma\in \mathbb{C}$, distinct from $\alpha$ and $\beta$. Then there exists a unique M\"obius transformation $f_\gamma$ such that
\[
f_\gamma(0)=\gamma, f_\gamma(\alpha)=\alpha, \text{ and } f_\gamma(\beta)=\beta.
\]
Using the coefficients of $f_\gamma$, we can uniquely define $\varphi_\gamma$ such that
\[
\tilde{\varphi_\gamma}=f_\gamma.
\]
Since $\tilde{\varphi}_\gamma$ fixes every element of $\clz(\theta)\setminus \{0\}$, by Theorem \ref{lem2.7}, each $\varphi_\gamma$ belongs to $L(\clq_\theta)^*$, and hence the set $L(\clq_\theta)^*$ is uncountable. Also, by Corollary \ref{cor: general group}, it follows that $L(\clq_\theta)^*$ is a group.

\noindent For part (2), pick $\varphi\in D(\clq_\theta)^*$. First, we assume that $m\neq n$. Corollary \ref{lary1} applies in particular to this case and asserts that $\tilde{\vp}^{[2]} = z$  (see also the proof of Corollary \ref{lary1}). That is, $\tilde{\varphi} \in Aut(\D)$ and is its own inverse. Then $\tilde{\varphi} \in Aut(\D)$ acts as a permutation on the set $\{\alpha,\beta\}$. Since $m\neq n$, Theorem \ref{lem2.7} implies that $\tilde{\varphi}$ fixes both $\alpha$ and $\beta$. However, a disc automorphism with two fixed points in $\D$ must be the identity. Hence, $\tilde{\varphi}$, and therefore $\varphi$, is identity. Next, assume that $m=n$. We assume that $\varphi\in D(\clq_\theta)^*$ and not the identity map. Then
\[
\tilde{\varphi}(\alpha)=\beta \text{ and } \tilde{\varphi}(\beta)=\alpha.
\]
However, there exists a unique disc automorphism with these properties, which we denote by $\tilde{b}_{\alpha,\beta}$. Thus,
\[
\varphi=b_{\alpha,\beta}.
\]
Therefore, in this case, $D(\clq_\theta)^*$ is a cyclic group of order $2$.
\end{proof}

The above result contradicts the conclusion of \cite[Theorem 2.8]{Jav}, which claims that even the larger set $L(\clq_\theta)$ is a finite cyclic group. However, the claim in \cite[Theorem 2.8]{Jav} is stated without proof. The claim in \cite[Theorem 2.8]{Jav}, however, admits a partial correction when the zero set is sufficiently large. The revised statement is as follows:

\begin{corollary}
Let $\theta$ be a finite Blaschke product. Suppose $\theta(0)=0$ and
\[
\#(\clz(\theta)\setminus \{0\}) \geq 3.
\]
Then the following two properties hold:
\begin{itemize}
\item[(a)] $L(\clq_\theta)^*$ and $D(\clq_\theta)^*$  are finite groups.
\item[(b)] If $\varphi\in L(\clq_\theta)^*$, then there exists $m \in \mathbb{N}$ such that  $\varphi^{[m]}=z$.
\end{itemize}
\end{corollary}
\begin{proof}
Set
\[
A:=\clz(\theta) \setminus \{0\}.
\]
Pick $\varphi \in L(\clq_\theta)^*$. By Corollary \ref{per}, $\tilde{\varphi}$ is a permutation on $A$, that is, $\tilde{\varphi}\in S_N$, where
\[
N=\#A.
\]
Since M\"obius maps are uniquely determined by the images of three distinct points, for each permutation $\sigma$ on $A$, there exists at most one M\"obius transformation $\tilde{\varphi}_\sigma$ such that
\[
\tilde{\varphi}_\sigma=\sigma \text{ on } A.
\]
Thus $L(\clq_\theta)^*$ is a finite set with at most $N(N-1)(N-2)$ elements. By Corollary \ref{lary1}, we know that the elements of $D(\clq_\theta)^*$ are disc automorphisms. Since a disc automorphism is uniquely determined by the images of two distinct points, $D(\clq_\theta)^*$ is a finite set with at most $N(N-1)$ elements. Hence, the proof of part (a) follows.

\noindent For (b), fix $\varphi \in L(\clq_\theta)^*$. By Corollary \ref{lary1}, there exist $m\in \mathbb{N}$ such that
\[
\tilde{\varphi}^{[m]}=z  \text{ on } A.
\]
Therefore the M\"obius map $\tilde{\varphi}^{[m]}$ has at least $\# A (\geq 3)$ fixed points, which implies that
\[
\tilde{\varphi}^{[m]}=z \mbox{~ on ~} \mathbb{C}_\infty.
\]
Then, by part (ii) of Lemma \ref{pro1}, it follows that ${\varphi}^{[m]}=z$.
\end{proof}

\begin{remark}
The conclusion of the above corollary also holds if we assume that $\#(\clz(\theta)\setminus \{0\}) \geq 2$ and $mult_\theta(0)\geq 2$. This is because $\tilde{\varphi}$ is a permutation of $\clz(\theta) \setminus \{0\}$ and $\tilde{\varphi}$ fixes the origin.
\end{remark}

\section{Concluding remarks}\label{sec: concluding}

In this concluding section, we offer some additional remarks on finite-dimensional model spaces that are invariant under composition operators. The following result, while elementary, is quite intriguing. Moreover, the same conclusion holds in the context of the space $L(\clq_\theta)$ (that is, for $C_\vp$ with $\vp \in L(\clq_\theta)$):

\begin{theorem}\label{theta2}
Let $\theta$ be a finite Blaschke product that does not vanish at the origin, and let $\vp \in \cls(\D)$. Then
\[
C_{\varphi}(\clq_\theta) \subseteq \clq_\theta,
\]
if and only if
\[
C_{\varphi}(\clq_\theta) = \clq_\theta.
\]
\end{theorem}
\begin{proof}
Suppose $C_{\varphi}(\clq_\theta) \subseteq \clq_\theta$. By Theorem \ref{thm1}, we know that $\vp$ is a rotation. Pick $f$ and $g$ from $\clq_\theta$ so that $C_{\varphi} f= C_{\varphi} g$, that is,
\[
f\circ \varphi= g \circ \varphi.
\]
As $\varphi$ is rotation, it follows that
\[
f \equiv g,
\]
and hence $C_\varphi$ is injective. Since $\clq_\theta$ is a finite-dimensional space, the rank-nullity theorem implies that $C_\vp$ is surjective.
\end{proof}

The situation changes when we assume that $\theta(0)=0$: For $\vp \in \cls(\D)$, assume that $C_{\varphi} (\clq_\theta)\subseteq \clq_\theta$. Then
\[
C_{\varphi} (\clq_\theta)=\clq_\theta,
\]
if and only if $\theta$ is a rotation. Indeed, if $\theta = z$, then $\clq_\theta$ is one dimensional, and $\clq_\theta= \mathbb{C}$. Then
\[
1 \circ \varphi= 1,
\]
for all $\vp \in \cls(\D)$, and hence $C_{\varphi}(\clq_\theta) = \clq_\theta$. For the converse direction, assume that $\theta$ is not a rotation. Given that $\theta(0)=0$, there exists a non-constant finite Blaschke product $B$ such that $\theta=zB$, which implies that
\[
\mathbb{C} \subsetneqq \clq_\theta.
\]
For a constant map $\varphi \in \cls(\D)$, as
\begin{equation*}
C_{\varphi} (\clq_\theta)=\mathbb{C}\subsetneq \clq_\theta,
\end{equation*}
we conclude that $C_{\varphi} (\clq_\theta) \neq \clq_\theta$. This proves that $\theta$ is a rotation implies $C_{\varphi} (\clq_\theta)=\clq_\theta$.

In particular, for $\theta(0)=0$, we have the following:
\begin{enumerate}
\item If $\theta$ is a rotation, then $C_{\varphi} (\clq_\theta)=\clq_\theta$ for all $\vp \in \cls(\D)$.
\item If $\theta$ is not a rotation, then there exists $\vp \in \cls(\D)$ such that $C_{\varphi} (\clq_\theta)\subseteq \clq_\theta$ and $C_{\varphi} (\clq_\theta)\neq \clq_\theta$.
\end{enumerate}

In \cite{Inv2015} (see also \cite{ISCO}), it was pointed out that, given a composition operator, there always exists a shift-invariant subspace that is invariant under the composition operator. A similar question can be posed for model spaces (that is, backward shift-invariant subspaces). The following result addresses this issue explicitly.

\begin{remark}
If $\theta=z$, then $\clq_\theta=\mathbb{C}$ is trivially invariant under every composition operator. Next, consider an analytic self-map $\varphi$ of $\mathbb{D}$ that is not a M\"obius map (for instance, $\varphi=z^2$). Since $\varphi$ is neither a rotation nor a  M\"obius map, it follows from the results of this paper that $\varphi\notin D(\clq_\theta)$ for all $\theta$ except when $\theta$ is a rotation. That is, $\clq_\theta$ is not invariant under $C_\varphi$ for any $\theta$ other than rotation. Hence, there exist many composition operators for which no non-trivial model space is invariant. That is, except for $\mathbb{C}$, all other model spaces fail to be invariant under any composition operator induced by a non-M\"obius maps.
\end{remark}

The results of this paper once again suggest that the theory of composition operators even when restricted to finite-dimensional model spaces exhibits significant variation from case to case. In particular, the results differ substantially between the cases of finite Blaschke products that vanish at the origin and those that do not. Even within these broad categories, the behavior further varies depending on specific subcases.

The following is a summary of some of the main results concerning finite-dimensional model spaces. These results are established in this paper. However, we reiterate that some of them are derived from \cite{Jav} and are presented here either verbatim (see also \cite{Mut}) or in a modified, corrected, or expanded form. In the following, $\theta$ denotes a finite Blaschke product, and $\clq_\theta$ refers to the corresponding finite-dimensional model space, with
\[
\dim \clq_\theta = \deg \theta.
\]
The first row of the following table specifies the conditions imposed on the function $\theta$, while the remaining rows outline the corresponding properties under each condition:

\begin{center}
\begin{tabular}{ |m{.5cm}| m{6cm}| m{8cm} | }
\hline
& $\theta(0)\neq 0$ & $\theta(0)=0$, $\theta\neq $ rotation \\
\hline
\hline
1 & $L(\clq_\theta)$ is a finite set & $L(\clq_\theta) \supseteq \mathbb{C}$, and hence uncountable \\
\hline
2 & $L(\clq_\theta)$ is a finite cyclic group & $L(\clq_\theta) \setminus \mathbb{C}$ is always a group\\
\hline
3 & Every element in $L(\clq_\theta)$ is of the form $\varphi(z)=az + b$ & Every element in $L(\clq_\theta)$ is either a constant or $\frac{az+b}{cz+d}$. \\
\hline
4 & In some cases, $L(\clq_\theta)=\{z\}$ (cf. Theorem \ref{thm - triv L}) & Always $L(\clq_\theta)\neq \{z\}$  \\
\hline
5 & $C_\vp (\clq_\theta)\subseteq \clq_\theta \Rightarrow C_\vp(\clq_\theta)=\clq_\theta$  &  $C_\vp(\clq_\theta)\subseteq \clq_\theta$ $\Rightarrow$
$C_\vp(\clq_\theta)\subsetneqq \clq_\theta$ \\
\hline
\end{tabular}
\end{center}

\bigskip

These results, along with the techniques used to obtain them, raise several questions of general interest. We conclude the paper by highlighting two particularly intriguing ones: The first question concerns the classification of finite-dimensional spaces that are invariant under composition operators. Specifically, given a finite-dimensional subspace $\cls \subseteq H^2(\D)$, classify all analytic self-maps $\varphi \in \cls(\D)$ such that
\[
C_\vp \cls  \subseteq \cls.
\]
The second question concerns infinite Blaschke products. In the case of finite Blaschke products that do not vanish at the origin, we observed that finite cyclic groups play a central role in characterizing the self-maps that leave the corresponding finite-dimensional model spaces invariant. Some of these results also make use of elementary tools, such as the prime factorization of natural numbers that arise as the cardinality of the zero set of finite Blaschke products. The natural next step is to explore this (or enhanced) phenomenon for model spaces associated with infinite Blaschke products that also do not vanish at the origin.

It is anticipated that group-theoretic tools will again be relevant, though precisely how they will be used remains an open and interesting problem. We expect a deeper interplay between group theory and analytic function theory to be essential in developing a complete understanding.

\bigskip

\noindent {\bf Acknowledgments:} The research of the second named author is supported in part by TARE (TAR/2022/000063) by SERB, Department of Science \& Technology (DST), Government of India. The third author is supported by a funded research project on ``Research in Operator Theory" (No. P2022-4370) by the National University of Mongolia.

\end{document}